\documentclass[a4,12pt]{amsart}
\usepackage{amssymb,amsmath,amsthm,txfonts}

\newtheorem{thm}{Theorem}[section]
\newtheorem{lem}[thm]{Lemma}

\newtheorem{prop}[thm]{Proposition}
\theoremstyle{definition}

\theoremstyle{remark}
\newtheorem{rem}[thm]{\textbf{Remark}}
\newtheorem{rems}[thm]{\textbf{Remarks}}

 \makeatletter
    
    \@addtoreset{equation}{section}
  \makeatother

\makeatletter
       \def\@makefnmark{%
               \leavevmode
               \raise.9ex\hbox{\check@mathfonts
                       \fontsize\sf@size\z@\normalfont%
                               \@thefnmark}%
       }
       \makeatother

\renewcommand\abstractname{\textbf{Abstract}}

\newcommand{\cal}{\mathcal}

\begin{document}

\title{Stokes Resolvent Estimates in Spaces of Bounded Functions}
\author{KEN ABE}
\address[K. ABE]{Graduate School of Mathematical Sciences University of Tokyo, Komaba 3-8-1, Meguro-ku, Tokyo, 153-8914, Japan}
\email{kabe@ms.u-tokyo.ac.jp}
\author{YOSHIKAZU GIGA }
\address[Y. GIGA]{Graduate School of Mathematical Sciences University of Tokyo, Komaba 3-8-1, Meguro-ku, Tokyo, 153-8914, Japan}    
\email{labgiga@ms.u-tokyo.ac.jp}
\author{Matthias Hieber}
\address[M. HIEBER]{Technische Universit\"at Darmstadt Fachbereich Mathematik Schlossgartenstr. 7, D-64289 Darmstadt, Germany and Center of Smart Interfaces Petersenstr. 32, D-64287 Darmstadt}
\email{hieber@mathematik.tu-darmstadt.de}
\keywords{Analytic semigroups, bounded function spaces, 
resolvent estimates . \\
\mbox{}{\hspace{.3cm}\textit{ Mots cl\'es.}}
Semi-groupes holomorphes, espace des fonctions born\'ees, estimation 
de la r\'esolvante
}

\date{}

\maketitle

\begin{abstract}
The Stokes equation on a domain $\Omega \subset \mathbf{R}^n$ is well
understood in the $L^p$-setting for a large class of domains including bounded and exterior domains with smooth boundaries provided $1<p<\infty$.
The situation is very different for the case $p=\infty$ since in this case the Helmholtz projection does  not act as a bounded  operator anymore. Nevertheless it was recently proved by the first and the second author
of this paper by a contradiction argument that the Stokes operator
generates an analytic semigroup on spaces of bounded functions for a
large class of domains.
 This paper presents a new approach as well as  new a
priori $L^\infty$-type estimates to the Stokes equation. They imply in particular that the Stokes operator generates a $C_0$-analytic semigroup of
angle $\pi/2$ on $C_{0,\sigma}(\Omega)$, or a non-$C_0$-analytic semigroup on $L^\infty_\sigma(\Omega)$ for a large class of domains. The approach presented is inspired by the so called Masuda-Stewart technique for elliptic operators. It is shown furthermore that the method presented applies also to different type of boundary conditions as, e.g., Robin boundary conditions.
\end{abstract}

\renewcommand\abstractname{\textbf{R\'esum\'e}}

\begin{abstract}
L'\'equation de Stokes sur un ouvert $\Omega \subset \mathbf{R}^n$ est bien compris dans le cadre de $L^p$ pour $1<p<\infty$ et pour une grande classe d'ouverts r\'eguliers.  La situation est bien diff\'erent pour le cas $p=\infty$, car la projection de Leray n'est pas born\'ee dans ce cas. Il \'etait d\'emontr\'e par le premier et second auteur de cet article que l'op\'erateur de Stokes quand-m\^eme engendre un semigroup holomorphe sur des espaces de
fonctions born\'ees pour une grande classe d'ouverts.
  Cet article pr\'esent une nouvelle approche  et des nouvelles \'estimations \`a priori de type $L^\infty$ pour l'equation de Stokes. Celles-ci impliquent
en particulier que l'op\'erateur de Stokes engendre un semigroup holomorphe d'angle $\pi/2$ sur $L^\infty_\sigma(\Omega)$ (pas fortement continue) ou $C_{0,\sigma}(\Omega)$
pour une grande classe d'ouverts $\Omega$. L'approche est inspir\'ee par la m\'ethode de Masuda-Stewart. Il est d\'emontr\'e de plus que la m\'ethode s'applique aussi \`a d'autres conditions au bord, par example les conditions de Robin.
\end{abstract}

\section{Introduction and main results}

The investigation of  the linear Stokes equations as well as properties and corresponding estimates are often basis for the analysis of the nonlinear Navier-Stokes equations. In particular, analyticity of the solution operator (called the Stokes semigroup) plays a fundamental role for studying the Navier-Stokes equations. It is well-known that the Stokes semigroup forms an analytic semigroup on $L^{p}_{\sigma}(\Omega)$ for $p\in (1,\infty)$, the space of $L^{p}$-solenoidal vector fields, for various kind of domains $\Omega\subset \mathbf{R}^{n}, n\geq2 $ including bounded and  exterior domains having smooth boundaries; see, e.g., \cite{Sl77}, \cite{Gg81}. By now, analyticity results are known for other type unbounded domains, see \cite{FS1}, \cite{FS}, \cite{AS} (\cite{AbT}, \cite{Ab} with variable viscosity coefficients) and Lipschitz domains \cite{Shen}. An $\tilde{L}^{p}$-theory is developed in \cite{FKS}, \cite{FKS1}, \cite{FKS2} for a general domain. Moreover, $L^{p}$-theory is investigated in \cite{GHHS} for unbounded domains, for which the Helmholtz projection is bounded.

It is the aim of this paper to consider the case $p=\infty$. Note that the Helmholtz projection is no longer bounded in $L^{\infty}$ even if $\Omega=\mathbf{R}^{n}$. When $\Omega = \mathbf{R}^n_+$, the analyticity of the semigroup is known in $L^{\infty}$-type spaces including $C_{0,\sigma}(\Omega)$, the $L^{\infty}$-closure of $C^{\infty}_{c,\sigma}(\Omega)$, the space of all smooth solenoidal vector fields compactly supported in $\Omega$ \cite{DHP}(see also \cite{Sl03}, \cite{MarS}). Their approach is based on explicit calculations of the solution operator $R(\lambda):f\mapsto v=v_{\lambda}$ of the 
corresponding resolvent problem of 
\begin{align*}
\lambda v - \Delta v  + \nabla q&= f \quad \mathrm{in}\ \  \Omega,   \tag{1.1}\\
\textrm{div}\ v&=0\quad \mathrm{in}\  \Omega,  \tag{1.2} \\
v&=0\quad \mathrm{on}\ \ \partial\Omega.   \tag{1.3}
\end{align*} 
As recently shown in \cite{AG1}, \cite{AG2} by a blow-up argument to the non-stationary Stokes equations, it turns out that the Stokes semigroup is extendable to an analytic semigroup on $C_{0,\sigma}$ for what is called \textit{admissible domains} which include bounded and exterior domains  having boundaries of class $C^3$. 

In this paper, we present a direct resolvent approach to the Stokes resolvent equations (1.1)--(1.3) and establish the a priori estimate of the form
\begin{equation*}
M_{p}(v,q)(x,\lambda)=|\lambda||v(x)|+|\lambda|^{1/2}|\nabla v(x)|+|\lambda|^{n/2p}||\nabla^{2} v||_{L^{p}(\Omega_{x,|\lambda|^{-1/2}})}+|\lambda|^{n/2p}||\nabla q||_{L^{p}(\Omega_{x,|\lambda|^{-1/2}})}, 
\end{equation*}
for $p>n$ and   
\begin{equation*}
\sup_{\lambda\in \Sigma_{\vartheta,\delta}}||M_{p}(v,q)||_{L^{\infty}(\Omega)}(\lambda)\leq C||f||_{L^{\infty}(\Omega)}   \tag{1.4}
\end{equation*}
for some constant $C>0$ independent of $f$. Here, $\Omega_{x,r}$ denotes the intersection of $\Omega$ with an open ball $B_{x}(r)$ centered at $x \in \Omega$ with radius $r>0$, i.e., $\Omega_{x,r}= B_{x}(r)\cap \Omega$ and $\Sigma_{\vartheta,\delta}$ denotes the sectorial region in the complex plane given by $\Sigma_{\vartheta,\delta}=\{\lambda\in \mathbf{C}\backslash \{0\}\ |\ |\arg{\lambda}|<\vartheta,\ |\lambda|>\delta\}$ for $\vartheta \in (\pi/2,\pi)$ and $\delta>0$. Our approach is inspired by the corresponding approach for general elliptic operators. K. Masuda was the first to prove analyticity of the semigroup associated to general elliptic operators in $C_{0}(\mathbf{R}^{n})$ including the case of higher orders \cite{Ma1}, \cite{Ma2} (\cite{Mas}.) This result was then extended by H. B. Stewart to the case for the Dirichlet problem \cite{Ste74} and more general boundary condition \cite{Ste80}. This Masuda-Stewart method was applied to many other situations \cite{AT},  \cite{L}, \cite{HHS}, \cite{Am}, \cite{APP}. However, its application to the resolvent Stokes equations (1.1)--(1.3) was unknown.  

In the sequel, we prove the estimate (1.4) by invoking the $L^{p}$-estimates for the Stokes resolvent equations with inhomogeneous divergence condition \cite{FS92}, \cite{FS1}. We invoke \textit{strictly admissibility} of a domain introduced in \cite[Definition 2.4]{AG2}  which implies an estimate of pressure $q$ in terms of the velocity by
\begin{equation*}
\sup_{x\in\Omega}d_\Omega(x)|\nabla q(x)| \leq C_\Omega  \tag{1.5}
\|W(v) \|_{L^\infty(\partial\Omega)} \quad\textrm{for}\ W(v)=-(\nabla v-\nabla^{T} v)n_{\Omega},
\end{equation*}
where $\nabla f$ denotes $(\partial f_{i}/\partial x_{j})_{1\leq i,j \leq n}$ and $\nabla^{T} f=(\nabla f)^{T}$ for a vector field $f=(f_{i})_{1\leq i\leq n}$. The estimate (1.5) plays a key role in transferring results from the elliptic situation to the  situation of the Stokes system. Here, $n_{\Omega}$ denotes the unit outward normal vector field on $\partial\Omega$ and $d_\Omega$ denotes the distance function from the boundary, i.e., $d_\Omega(x)=\inf_{y \in \partial\Omega}|x-y|$ for $x \in \Omega$. The estimate (1.5) can be viewed as a regularizing-type estimate for solutions to the Laplace equation $\Delta P=0$ in $\Omega$ with the Neumann boundary condition $\partial P/\partial n_{\Omega}=\textrm{div}_{\partial\Omega} W$ on $\partial\Omega$ for a tangential vector field $W$, where $\textrm{div}_{\partial\Omega}=\textrm{tr}\ \nabla _{\partial\Omega}$ denotes the surface divergence and $\nabla_{\partial\Omega}=\nabla-n_{\Omega}(n_{\Omega}\cdot \nabla)$ is the gradient on $\partial\Omega$. It is known that $P=q$ solves this Neumann problem for $W=W(v)$ given by (1.5) \cite[Lemma 2.8]{AG2} and the estimate (1.5) holds for bounded domains \cite{AG1} and exterior domains \cite{AG2}. Note that when $n=3$, $W(v)$ is nothing but a tangential trace of vorticity, i.e., $W(v)=-\textrm{curl}\ v\times n_{\Omega}$. We call $\Omega$ \textit{strictly admissible} if there exists a constant $C=C_{\Omega}$ such that the a priori estimate 
\begin{equation*}
||\nabla P||_{L^{\infty}_{d}(\Omega)}\leq C ||W||_{L^{\infty}(\partial\Omega)}  \tag{1.6}
\end{equation*}
 holds for all solutions $P$ of the Neumann problem for a tangential vector field $W\in L^{\infty}(\partial\Omega)$. Here $L^{\infty}_{d}(\Omega)$ denotes the space of all locally integrable functions $f$ such that $d_{\Omega}f$ is essentially bounded in $\Omega$ and  equipped with the norm $||f||_{L^{\infty}_{d}(\Omega)}=\sup_{x\in \Omega}d_{\Omega}(x)|f(x)|$. The meaning of a solution is understood in the weak sense, i.e., we say $\nabla P\in L^{\infty}_{d}(\Omega)$ is a solution for the Neumann problem if $\int_{\Omega}P\Delta \varphi dx=\int_{\partial\Omega} W\cdot\nabla_{\partial\Omega}\varphi d{{\cal H}^{n-1}}(x)$ holds for all $\varphi\in C^{2}_{c}(\bar\Omega)$ satisfying $\partial\varphi/\partial n_{\Omega}=0$ on $\partial\Omega$, where ${{\cal H}^{n-1}}$ denotes the $n-1$-dimensional Hausdorff measure; see also  \cite[Definition 2.3]{AG2}.\\

We are now in the position to formulate the main results of this paper.

\begin{thm}
Let $\Omega$ be a  strictly admissible, uniformly $C^{2}$-domain in $\mathbf{R}^{n}$, $n\geq 2$. Let $p>n$. For $\vartheta\in (\pi/2,\pi)$, there exists constants $\delta$ and $C$ such that the a priori estimate (1.4) holds for all solutions $(v,\nabla q)\in W^{2,p}_{\textrm{loc}}(\bar{\Omega})\times (L^{p}_{\textrm{loc}}(\bar{\Omega})\cap L^{\infty}_{d}(\Omega))$ 
of (1.1)--(1.3) for $f\in C_{0,\sigma}(\Omega)$ and $\lambda\in \Sigma_{\vartheta,\delta}$.    
\end{thm}

The a priori estimate (1.4) implies the analyticity of the Stokes semigroup in $L^{\infty}$-type spaces. Let us observe the generation of an analytic semigroup in $C_{0,\sigma}(\Omega)$. By invoking the $\tilde{L}^{p}$-theory \cite{FKS}, \cite{FKS1}, \cite{FKS2} we verify the existence of a solution to (1.1)--(1.3), $(v,\nabla q)\in W^{2,p}_{\textrm{loc}}(\bar{\Omega})\times (L^{p}_{\textrm{loc}}(\bar{\Omega})\cap L^{\infty}_{d}(\Omega))$ for $f\in C^{\infty}_{c,\sigma}(\Omega)$ in a uniformly $C^{2}$-domain $\Omega$. The solution operator $R(\lambda)$ is then uniquely extendable to $C_{0,\sigma}(\Omega)$ by the uniform approximation together with the estimates (1.4). Here, the solution operator to the pressure gradient $f\mapsto \nabla q_{\lambda}$ is also uniquely extended for  $f\in C_{0,\sigma}$. We observe that $R(\lambda)$ is injective on $C_{0,\sigma}$ since the estimate (1.5) immediately implies that $f=0$ for $f\in C_{0,\sigma}$ such that $v_{\lambda}=R(\lambda)f=0$. The operator $R(\lambda)$ may be regarded as a surjective operator from $C_{0,\sigma}$ to the range of $R(\lambda)$. The open mapping theorem then implies the existence of a closed operator $A$ such that $R(\lambda)=(\lambda-A)^{-1}$; see \cite[Proposition B.6]{ABHN}. We call $A$ \textit{the Stokes operator} in $C_{0,\sigma}(\Omega)$. From Theorem 1.1, we obtain:  

\begin{thm}
Let $\Omega$ be a  strictly admissible, uniformly $C^{2}$-domain in $\mathbf{R}^{n}$. Then, the Stokes operator $A$ generates a $C_{0}$-analytic semigroup on $C_{0,\sigma}(\Omega)$ of angle $\pi/2$.
\end{thm}

We next consider the space $L^{\infty}_{\sigma}(\Omega)$ defined by
\begin{equation*}
L^{\infty}_{\sigma}(\Omega)=\left\{f\in L^{\infty}(\Omega)\ \Bigg|\ \int_{\Omega}f\cdot \nabla \varphi dx=0\quad \textrm{for all}\ \varphi\in \hat{W}^{1,1}(\Omega)\right\},
\end{equation*}
where $\hat{W}^{1,1}(\Omega)$ denotes the homogeneous Sobolev space of the form $\hat{W}^{1,1}(\Omega)=\{\varphi\in L^{1}_{\textrm{loc}}(\Omega)\ |\ \nabla\varphi\in L^{1}(\Omega)\}$. Note that $C_{0,\sigma}(\Omega)\subset L^{\infty}_{\sigma}(\Omega)$. When the domain $\Omega$ is unbounded, the space $L^{\infty}_{\sigma}(\Omega)$ includes non-decaying solenoidal vector fields at the space infinity. Actually, the a priori estimates (1.4) is also valid for $f\in L^{\infty}_{\sigma}$. In particular, (1.4) implies the uniqueness of a solution for $f\in L^{\infty}_{\sigma}$. We verify the existence of a solution by  approximating $f\in L^{\infty}_{\sigma}$ with compactly supported solenoidal vector fields $\{f_{m}\}_{m=1}^{\infty}\subset C^{\infty}_{c,\sigma}$. Note that $f\in L^{\infty}_{\sigma}$ is not  approximated in the uniform topology by the elements of $C^{\infty}_{c,\sigma}$ in general. We thus weaken the convergence topology to the pointwise convergence, i.e., $f_m\to f$ a.e. in $\Omega$ and $||f_{m}||_{L^{\infty}(\Omega)}\leq C||f||_{L^{\infty}(\Omega)}$ with some constant $C=C_{\Omega}$. When the domain $\Omega$ is bounded, this approximation is valid \cite[Lemma 6.3]{AG1}. Although this approximation is unknown in general, for exterior domains, it is known to hold \cite[Lemma 5.1]{AG2}. In the following, we restrict our results to bounded and exterior domains. By the approximation argument for $L^{\infty}_{\sigma}$, we verify the existence of a solution to (1.1)--(1.3) for general $f\in L^{\infty}_{\sigma}$. We then define the Stokes operator on $L^{\infty}_{\sigma}$ by the same way as for $C_{0,\sigma}$. Since bounded and exterior domains are strictly admissible \cite[Theorem 2.5]{AG1}, \cite[Theorem 3.1]{AG2} provided that the boundary is $C^{3}$, we have:

\begin{thm}
Assume that $\Omega$ is a bounded or an exterior domain with $C^{3}$-boundary. Then, the Stokes operator $A$ generates a (non-$C_0$-)analytic semigroup on $L^{\infty}_{\sigma}(\Omega)$ of angle $\pi/2$.
\end{thm}

\begin{rems}\normalfont
(i) The direct resolvent approach clarifies the angle of the analyticity of the Stokes semigroup $e^{tA}$ on $C_{0,\sigma}$. Theorem 1.2 (and also Theorem 1.3) asserts that $e^{tA}$ is angle $\pi/2$ on $C_{0,\sigma}$ which does not follow from a priori $L^{\infty}$-estimates for solutions to the non-stationary Stokes equations proved by blow-up arguments \cite[Theorem 1.2]{AG1}, \cite[Lemma 2.12]{AG2}.\\  
\noindent
(ii) We observe that our argument applies to other boundary conditions, for example, to the Robin boundary condition, i.e., $B(v)=0$ and $v\cdot n_{\Omega}=0$ on $\partial\Omega$ where
\begin{equation*}
B(v)=\alpha v_{\textrm{tan}}+(D(v)n_{\Omega})_{\textrm{tan}} \quad \textrm{for}\ \alpha\geq0.
\end{equation*} 
Here, $D(v)=(\nabla v+\nabla^{T}v)/2$ denotes the deformation tensor and $f_{\textrm{tan}}$ the tangential component of a  vector field $f$ on $\partial\Omega$. Note that the case $\alpha=\infty$ corresponds to the Dirichlet boundary condition (1.3); see \cite{Sa} for generation results subject to the Robin boundary conditions on $L^{\infty}$ for $\mathbf{R}^{n}_{+}$. The $L^{p}$-resolvent estimates for the Robin boundary condition was established in \cite{Gg82} concerning analyticity and was later strengthened in \cite{ShbS} to non homogeneous divergence vector fields. We shall use the 
generalized resolvent estimate in \cite{ShbS} to extend our result in spaces of bounded functions to the Robin boundary condition (Theorem 3.6). For a more detailed discussion, see Remark 3.5.\\
(iii) We observe that the domain of the Stokes operator $D(A)$ is dense in $C_{0,\sigma}$. In fact, by the $\tilde{L}^{p}$-theory and (1.4), we have
\begin{equation*}
||\lambda v-f||_{L^{\infty}(\Omega)}
=||\tilde{A}_{p}v||_{L^{\infty}(\Omega)}
\leq \frac{C}{|\lambda|}||\tilde{A}_{p}f||_{L^{\infty}(\Omega)}\to 0,\quad |\lambda|\to \infty
\end{equation*} 
for $f\in C^{\infty}_{c,\sigma}\subset D(\tilde{A}_{p})$, where $\tilde{A}_{p}$ is the Stokes operator in $\tilde{L}^{p}$. Thus, we conclude that $D(A)$ is dense in $C_{0,\sigma}$. On the other hand, smooth functions are not dense in $L^{\infty}$ and $e^{tA}f$ is smooth for $t>0$, $e^{tA}f\to f$ as $t\downarrow 0$ in $L^{\infty}_{\sigma}$ does not hold for some $f\in L^{\infty}_{\sigma}$. This means $e^{tA}$ is a non-$C_0$-analytic semigroup. We refer to \cite[1.1.2]{Sin85} for properties of the analytic semigroup generated by non-densely defined sectorial operators; see also \cite[Definition 3.2.5]{ABHN}.\\     
\noindent
(iv) For a bounded domain $\Omega$, $v(\cdot,t)=e^{tA}v_0$ and $\nabla q=(1-\bold{P})[\Delta v]$ give a solution to the non-stationary Stokes equations, $v_{t}-\Delta v+\nabla q=0,\ \textrm{div}\ v=0$ in $\Omega\times (0,\infty)$ with $v=0$ on $\partial\Omega$ for initial data $v_0\in L^{\infty}_{\sigma}(\Omega)$. Although for unbounded domains the Helmholtz projection operator $\bold{P}: L^{p}(\Omega)\to L^{p}_{\sigma}(\Omega)$ is not bounded on $L^{\infty}$ even for $\Omega=\mathbf{R}^{n}$, we are able to define the pressure $\nabla q=\bold{K}[W(v)]$ at least for exterior domains $\Omega$ by the solution operator to the Neumann problem (\textit{harmonic-pressure operator}) $\bold{K}:L^{\infty}_{\textrm{tan}}(\partial\Omega)\ni W\mapsto \nabla P \in L^{\infty}_{d}(\Omega)$ \cite[Remarks 4.3 (ii)]{AG2}. Here, $L^{\infty}_{\textrm{tan}}(\partial\Omega)$ denotes the closed subspace of all tangential vector fields in $L^{\infty}(\partial\Omega)$. \\
\noindent 
(v) We observe that the Masuda-Stewart method does not imply the  large time behavior for $e^{tA}$. For a bounded domain, the energy inequality implies that maximum of $v(\cdot,t)=e^{tA}v_0$ (and also $v_{t}$) decay exponentially as $t\to \infty$ \cite[Remark 5.4 (i)]{AG1}. In particular, $e^{tA}$ is a bounded analytic semigroup on $L^{\infty}_{\sigma}$. Recently, based on the $L^{\infty}$-estimates \cite[Theorem 1.2]{AG1} it was shown in \cite{Mar12} that $e^{tA}$ is a bounded semigroup on $L^{\infty}_{\sigma}$ for $\Omega$ being an exterior domain with smooth boundary. 
\end{rems}

\noindent
In the sequel, we sketch a proof for the a priori estimate (1.4). Our argument can be divided into the following three steps:\\

\noindent
(i) (Localization) We first localize a solution $(v,q)$ of the Stokes equations (1.1)--(1.3) in a domain $\Omega'=B_{x_0}((\eta+1)r)\cap \Omega$ for $x_0\in \Omega, r>0$ and parameters $\eta\geq 1$ by setting $u=v\theta_{0}$ and $p=(q-q_c)\theta_{0}$ with a constant $q_c$ and the smooth cut-off function $\theta_0$ around $\Omega_{x_0,r}$ satisfying $\theta_0\equiv 1$ in $B_{x_0}(r)$ and $\theta_0\equiv 0$ in $B_{x_0}((\eta+1)r)^{c}$. We choose parameters $\eta\geq 1$ and $r>0$ so that $(\eta+2)r\leq r_0$ with some constant $r_0$. We then observe that $(u,p)$ solves the Stokes resolvent equations with inhomogeneous divergence condition in the localized domain $\Omega'$. Since we adjust parameters $\eta\geq 1$ later, we take a $C^{2}$-bounded domain $\Omega''$ so that $\Omega_{x_0,r_{0}}\subset \Omega''\subset \Omega$. Then, $\Omega'\subset \Omega''$ for all $\eta\geq 1$ and $r>0$ satisfying $(\eta+2)r\leq r_0$. We apply the $L^{p}$-estimate for the localized Stokes equations in $\Omega''$ to get 
\begin{align*}
&|\lambda| ||u||_{L^{p}(\Omega'')}+|\lambda|^{1/2} ||\nabla u||_{L^{p}(\Omega'')}+||\nabla^{2}u||_{L^{p}(\Omega'')}+||\nabla p||_{L^{p}(\Omega'')}\\
&\leq C_{p}\left(||h||_{L^{p}(\Omega'')}+||\nabla g||_{L^{p}(\Omega'')}+|\lambda| ||g||_{W^{-1,p}_{0}(\Omega'')}\right), \tag{1.7} 
\end{align*} 
where $W^{-1,p}_{0}(\Omega'')$ denotes the dual space of the Sobolev space $W^{1,p'}(\Omega'')$ with $1/p+1/p'=1$. The constant $C_{p}$ depends on $r_0$ and a choice of $\Omega''$ but is independent of $\eta\geq 1$ and $r>0$ satisfying $(\eta+2)r\leq r_0$. The external forces $h$ and $g$ contain error terms appearing in the cut-off procedure and are explicitly given by 
\begin{equation*}
h=f\theta_{0}-2\nabla v\nabla\theta_{0}-v\Delta\theta_{0}+(q-q_c)\nabla \theta_{0},\quad  g=v\cdot\nabla \theta_{0}.  \tag{1.8}
\end{equation*}       
\noindent
(ii) (Error estimates) A key step is to estimate the error terms of the  pressure such as $(q-q_c)\nabla \theta_0$. We here simplify the description by disregarding the terms related to $g$ in order to describe the essence of the proof. We will give precise estimates for the terms related to $g$ in Section 3. Now, the error terms related to $h$ supported in $\Omega'$ are estimated in the form 
\begin{equation*}
||h||_{L^{p}(\Omega')}
\leq Cr^{n/p}\Bigg((\eta+1)^{n/p}||f||_{L^{\infty}(\Omega)}
+(\eta+1)^{-(1-n/p)}\Big(r^{-2}||v||_{L^{\infty}(\Omega)}+r^{-1}|| \nabla v||_{L^{\infty}(\Omega)}\Big)\Bigg).   \tag{1.9}
\end{equation*}
If we disregard the term $(q-q_{c})\nabla\theta_{0}$ in $h$, the estimates (1.8) easily follows by using the estimates of the cut-off function $\theta_{0}$, i.e., $||\theta_{0}||_{\infty}+(\eta+1)r||\nabla \theta_{0}||_{\infty}+(\eta+1)^{2}r^{2}||\nabla^2\theta_0||_{\infty}\leq K$ with some constant $K$. We invoke the estimate (1.5) in order to estimate the pressure term by velocity term through the \textit{Poincar\'{e}-Sobolev-type inequality}:   
\begin{equation*}
||\varphi-(\varphi)||_{L^{p}(\Omega_{x_0,s})}\leq C{s}^{n/p}||\nabla \varphi||_{L^{\infty}_{d}(\Omega)}\quad \textrm{for all}\ \varphi \in \hat{W}^{1,\infty}_{d}(\Omega),   \tag{1.10}
\end{equation*}
with some constant $C$ independent of $s\leq r_0$, where $(\varphi)$ denotes the mean value of $\varphi$ in $\Omega_{x_0,s}$ and $\hat{W}^{1,\infty}_{d}(\Omega)=\{\varphi\in L^{1}_{\textrm{loc}}(\bar\Omega)\ |\ \nabla \varphi\in L^{\infty}_{d}(\Omega) \}$. We prove the inequality (1.10) in Section 2. By taking $q_c=(q)$ and applying (1.10) for $\varphi =q$ and $s=(\eta+1)r$, we obtain the estimate (1.9) via (1.5). \\
\noindent     
(iii) (Interpolation) Once we establish the error estimates for $h$ and $g$, it is easy to obtain the estimate (1.4) by applying the interpolation inequality,
\begin{equation*}
||\varphi||_{L^{\infty}(\Omega_{x_0,r})}\leq
 C_{I}r^{-n/p} \left(||\varphi||_{L^{p}(\Omega_{x_0,2r})}+r||\nabla \varphi||_{L^{p}(\Omega_{x_0,2r})}\right)  \quad\textrm{for}\ \varphi\in W^{1,p}_{\textrm{loc}}(\bar\Omega),      \tag{1.11}
\end{equation*}
and $x_0\in \Omega$, $r\leq r_0$. The constant $C_I$ is independent of $x_0$ and $r$. We give a proof for the inequality (1.11) in Appendix A. Applying the above inequality for $\varphi=u$ and $\nabla u$ and now taking $r=|\lambda|^{-1/2}$, we obtain the estimate for $M_{p}(v,q)(x_0,\lambda)$ with the parameters $\eta$ of the form, 
\begin{equation*}
M_{p}(v,q)(x_0,\lambda)\leq C\left((\eta+1)^{n/p}||f||_{L^{\infty}(\Omega)}+(\eta+1)^{-(1-n/p)}||M_{p}(v,q)||_{L^{\infty}(\Omega)}(\lambda)\right)  \tag{1.12}
\end{equation*}
for some constant $C$ independent of $\eta$. Note that $r=|\lambda|^{-1/2}$ and $\eta$ satisfy $r(\eta+2)\leq r_0$ for all $\eta\geq 1$ and $|\lambda|\geq \delta_{\eta}$ where $\delta_{\eta}=(\eta+2)^{2}/r_0^{2}$. The second term in the right-hand side is absorbed into the left-hand side by letting $\eta$ sufficiently large provided $p>n$.\\

Actually, in the procedure (ii) we take $q_{c}$ by the mean value of $q$ in $\Omega_{x_0,(\eta+2)r}$ and apply the inequality (1.10) for $s=(\eta+2)r$ since we estimate $|\lambda|||g||_{W^{-1,p}_{0}}$. By using the equation (1.1), we reduce the estimate of $|\lambda|||g||_{W^{-1,p}_{0}}$ to the $L^{\infty}$-estimate for the boundary value of $q-q_c$ on $\partial\Omega'$. In order to estimate $||q-q_c||_{L^{\infty}(\Omega')}$, we use a uniformly local $L^{p}$-norm bound for $\nabla q$ besides the sup-bound for $\nabla v$. This is the reason why we need the norm $||M_{p}(v,q)||_{L^{\infty}(\Omega)}(\lambda)$ in the right-hand side of (1.12). For general elliptic operators, the estimate (1.12) is valid  without invoking the uniformly local $L^{p}$-norm bound for second derivatives of solutions.

This paper is organized as follows. In Section 2, we prove the inequality (1.10) for uniformly $C^{2}$-domains. More precisely, we prove stronger estimates than (1.10) both interior and up to boundary $\Omega_{x_0,s}$ of $\Omega$. In Section 3, we first prepare the estimates for $h$ and $g$ and then prove the a priori estimate (1.4) (Theorem 1.1). After proving Theorem 1.1, we also discuss the estimates (1.4) under the Robin boundary condition. \\

\begin{rems}
(i) After this work was completed, it turned out that a perturbed half space of class $C^{3}$ for $n\geq 3$ was also strictly admissible \cite[Theorem 2.3.3]{A}. Furthermore, the approximation for $f\in L^{\infty}_{\sigma}$ by $\{f_m\}_{m=1}^{\infty}\subset C^{\infty}_{c,\sigma}$, i.e., $f_m\to f $ a.e. in $\Omega$ and $||f_m||_{\infty}\leq C||f_{\infty}||_{\infty}$ with $C=C_{\Omega}$, was proved for a perturbed half space, $n\geq 2$, in \cite[Lemma 4.3.10]{A}. Thus, our main theorems (Theorem 1.1--Theorem 1.3) are also valid for a perturbed half space with $C^{3}$-boundary for $n\geq 3$.\\
\noindent
(ii) After this work was completed, the authors were informed of the recent paper by Kenig et al. \cite{KLS}, where the estimate (1.6) was proved for $C^{1,\gamma}$-bounded domains by estimating the Green function for the Neumann problem (independently of the works \cite{AG1}, \cite{AG2}, \cite{A}). If one applies their result, one is able to reduce the regularity assumption of  boundaries from $C^3$ to $C^2$ at least for bounded domains; the assertion of Theorem 1.3  is still valid for bounded domains with $C^2$-boundary. For elliptic operators, the estimate corresponding to (1.4) is valid with $C^{1,1}$-boundary. However, we use the $C^{2}$-regularity in the proof of the inequality (1.10). Note that $C^{1,1}$-boundary is sufficient for the $L^{p}$-estimate of the Stokes equations (1.7); see \cite{FS1}.\\
\noindent
(iii) After this work was completed, it was proved in \cite{HM} that $e^{tA}$ is a bounded analytic semigroup on $L^{\infty}_{\sigma}(\Omega)$, provided that $\Omega$ is an exterior domain with smooth boundary.
\end{rems}

\section{Poincar\'{e}-Sobolev-type inequality }

In this section, we prove the inequality (1.10) in a uniformly $C^{2}$-domain. We start with the Poincar\'{e}-Sobolev-type inequality in a bounded domain $D$ and observe the compactness of the  embedding from $\hat{W}^{1,\infty}_{d}(D)$ to $L^{p}(D)$ which is the key in proving the inequality (1.10) by \textit{reductio ad absurdum}.         

\subsection{Curvilinear coodinates}

Let $D$ be a bounded domain in $\mathbf{R}^{n}, n\geq 2$ and $p\in [1,\infty)$. We prove an inequality of the form,  
 \begin{equation*}
||\varphi-(\varphi)||_{L^{p}(D)}\leq C||\nabla \varphi||_{L^{\infty}_{d}(D)}\quad \textrm{for}\ \varphi\in \hat{W}^{1,\infty}_{d}(D)  \tag{2.1}
\end{equation*}
where $(\varphi)$ denotes the mean value of $\varphi$ in $D$, i.e., $(\varphi)=\fint_{D}\varphi dx$. If we replace the norm $||\nabla \varphi||_{L^{\infty}_{d}(D)}$ by the $L^{p}$-norm $||\nabla \varphi||_{L^{p}(D)}$, the estimate (2.1) is nothing but the Poincar\'{e} inequality \cite[5.8.1]{E}.  We observe that the boundedness of $||\nabla \varphi||_{L^{\infty}_{d}(\Omega)}$ implies $L^{p}$-integrability of $\varphi$ in $D$ even if $\nabla \varphi$ is not in $L^{p}(D)$. For example, when $D=B_{0}(1)$, $\varphi(x)=\log{(1-|x|)}$ is in $L^{p}$ although $|\nabla \varphi(x)|=d_{D}(x)^{-1}$ is not for any $p\in [1,\infty)$. Since the space $\hat{W}^{1,\infty}_{d}$ is compactly embedded to the space $C(\bar{D'})$ for each subdomain $D'$ of $D$ with $\bar{D'}\subset D$, we shall show a pointwise upper bound for $\varphi$ near $\partial D'$ by an $L^{p}$-integrable function to conclude that the space $\hat{W}^{1,\infty}_{d}(D)$ is compactly embedded to $L^{p}(D)$ by the dominated convergence theorem. We estimate $\varphi\in \hat{W}^{1,\infty}_{d}(D)$ near $\partial D$ directly by using the curvilinear coordinates. Here, for a domain $\Omega, \partial\Omega\neq \emptyset$, we say that $\partial \Omega$ is $C^{k}$ if for each $x_{0}\in \partial \Omega$, there exists constants $\alpha,\beta$ and $C^{k}$-function $h$ of $n-1$ variables $y'$ such that (up to rotation and translation if necessary) we have  
\begin{align*} 
&U(x_0) \cap \Omega=\bigl\{(y',y_n) \bigm| h(y')<y_n<h(y')+\beta,\ |y'|<\alpha\bigr\},   \\ 
&U(x_0) \cap \partial\Omega=\bigl\{(y',y_n) \bigm| y_n=h(y'),|y'|<\alpha\bigr\},\\
&\sup\limits_{|l|\leq k, |y'|<\alpha}\bigl|\partial^{l}_{y'} h(y')\bigr| \leq K,\  \nabla 'h(0)=0,\ h(0)=0,
\end{align*}
with the constant $K$ and the neighborhood of $x_0$, $U(x_0)=U_{\alpha,\beta,h}(x_0)$, i.e.,
\begin{equation*} 
U_{\alpha,\beta,h}(x_0)=\bigl\{(y',y_n) \in \mathbf{R}^n \bigm| h(y')-\beta<y_n<h(y')+\beta, |y'|<\alpha\bigr\}.
\end{equation*}
Here, $\partial_x^l=\partial_{x_1}^{l_1} \cdots \partial_{x_n}^{l_n}$ for  a multi-index $l=(l_1, \ldots, l_n)$ and $\partial_{x_j}=\partial/\partial x_j$ as usual and $\nabla '$ denotes the gradient in $\mathbf{R}^{n-1}$. Moreover, if we are able to take uniform constants $\alpha,\beta, K$ independent of each $x_0\in \partial\Omega$, we call $\Omega$ uniformly $C^{k}$-domain of type $(\alpha,\beta,K)$ as defined in \cite[I.3.2]{Sh}.\\

We estimate $\varphi\in \hat{W}^{1,1}_{d}(\Omega)$ along the boundary using the curvilinear coordinates. 

\begin{prop}
Let $D$ be a bounded domain with $C^{k}$-boundary, $k\geq 2$. Let $\Gamma=\{x\in \partial D \ |\ x=(x',h(x')),|x'|<\alpha'\}$ be a neighborhood of $x_0\in \partial D$.\\ 
(i) There exists positive constants $\mu$ and $\alpha'$ such that $(\gamma,d)\mapsto X(\gamma,d)=\gamma+dn_{D}(\gamma)$ is a  $C^{k-1}$ diffeomorphism from $\Gamma\times (0,\mu)$ onto 
\begin{equation*}
{\cal{N}^{\mu}}(\Gamma)=\{X(\gamma,d)\in U(x_0)\ |\ (\gamma,d)\in \Gamma \times (0,\mu)\},
\end{equation*}
i.e., $x\in {\cal{N}}^{\mu}(\Gamma)$ has a unique projection to $\partial D$ denoted by $\gamma(x)\in \partial D$ such that 
\begin{equation*}
(\gamma(x), d_{D}(x))=X^{-1}(x)\quad \textrm{for}\ x\in {\cal{N}^{\mu}}(\Gamma).
\end{equation*}

\noindent
(ii) There exists a constant $C_1$ such that for any $x_1\in \overline{{\cal{N}^{\mu}}(\Gamma)}$ and $r_1>0$ satisfying $D_{x_1,r_1}=B_{x_1}(r_1)\cap D\subset {\cal{N}^{\mu}}(\Gamma)$, 
\begin{equation*}
|\varphi(x)-\varphi(y)|\leq C_1\left(\left|\log{\frac{d_{D}(x)}{d_{D}(y)}}\right|+\frac{|\gamma(x)-\gamma(y)|}{\max\{d_{D}(x), d_{D}(y)\}} \right)\sup_{z\in D_{x_1,r}}d_{D}(z)|\nabla \varphi(z)| \quad  \textrm{for}\ x,\ y\in D_{x_1,r_1}  
\end{equation*}
and $\varphi\in \hat{W}^{1,\infty}_{d}(D)$.
\end{prop}

\begin{proof}
The assertion (i) is based on the inverse function theorem \cite[Lemma 4.4.7]{KP}. We shall prove the second assertion (ii). We take points $x,y\in D_{x_1,r_1}$ for $x_{1}\in \overline{{\cal{N}^{\mu}}(\Gamma)}$ and $r_1>0$ satisfying $D_{x_1,r_1}\subset {\cal{N}}^{\mu}(\Gamma)$. We may assume $d_{D}(y)=d(y)>d(x)$. By setting $z=X(\gamma (x), d(y))$ we estimate 
\begin{equation*}
|\varphi(x)-\varphi(y)|\leq |\varphi(x)-\varphi(z)|+|\varphi(z)-\varphi(y)|.
\end{equation*}
We connect $x$ and $z$ by the straight line to estimate 
\begin{equation*}
\begin{split}
|\varphi(x)-\varphi(z)|
&=\left|\int_{0}^{1}\frac{d}{dt}\varphi(X(\gamma (x), td(x)+(1-t)d(y)))dt\right|\\
&=\left|\int_{0}^{1}(d(y)-d(x))(\nabla \varphi)(X(\gamma (x), td(x)+(1-t)d(y))\cdot n_{D}(\gamma (x))dt\right|\\
&\leq (d(y)-d(x))\int_{0}^{1}\frac{dt}{t(d(x)-d(y))+d(y)}\sup_{z\in D_{x_1,r}}d(z)|\nabla \varphi(z)|\\
&=\left|\log{\frac{d(y)}{d(x)}}\right| \sup_{z\in D_{x_1,r}}d(z)|\nabla \varphi(z)|.
\end{split}
\end{equation*}
It remains to estimate $|\varphi(z)-\varphi(y)|.$ We connect $z$ and $y$ by the curve 
\begin{equation*}
C_{z,y}=\{X(\gamma(t),d(y))\ |\ \gamma(t)=(\gamma'(t),h(\gamma'(t))),\ \gamma'(t)=t\gamma'(x)+(1-t)\gamma'(y),\ 0\leq t\leq 1 \},
\end{equation*}
where $\gamma'$ denotes the $n-1$ variables of $\gamma$. We then estimate  
\begin{equation*}
\begin{split}
|\varphi(z)-\varphi(y)|
&=\left|\int_{0}^{1}\frac{d}{dt}\varphi(X(\gamma(t),d(y)))dt\right|\\
&=\left|\int_{0}^{1}\frac{d \gamma}{dt}(t)(1+d(y)\nabla_{\partial D}n_{D}(\gamma (t)))\nabla\varphi(X(\gamma(t),d(y)))dt\right|\\
&\leq C(1+\mu K)\frac{|\gamma(x)-\gamma(y)|}{d(y)}\sup_{z\in D_{x_1,r_1}}d(z)|\nabla \varphi(z)|,
\end{split}
\end{equation*}
since $|d\gamma (t)/dt|\leq C|\gamma(x)-\gamma(y)|$ and $|\nabla_{\partial D}n_{D}|\leq K$ with a constant $C$ depending on $K$. The assertion (ii) thus follows. 
\end{proof}

\begin{rems}\normalfont
(i) We observe from the second assertion that $\varphi\in \hat{W}^{1,\infty}_{d}(D)$ is bounded from above by an $L^{p}$-integrable function for all $p\in [1,\infty)$ near $\partial D$, i.e., for each fixed $y\in D_{x_1,r_1}$ such that $d_{D}(y)\geq \delta$ we have
\begin{equation*}
|\varphi(x)|\leq C_2(|\log{d_{D}(x)}|+1)\left(\sup_{z\in D_{x_1,r_1}}d_{D}(z)|\nabla \varphi(z)|\right)+|\varphi(y)|\quad\textrm{for}\ x\in D_{x_1,r_1}    \tag{2.2}
\end{equation*}
with a constant $C_2$ depending on $\mu, \delta$. 

\noindent 
(ii) Note that Proposition 2.1 is also valid for a uniformly $C^{k}$-domain $\Omega$ of type $(\alpha,\beta,K)$, i.e., there exists constants $\mu, \alpha'$, depending only on $\alpha,\beta,K$, such that for each $x_0\in \partial\Omega$ the assertions (i) and (ii) hold.  The above constants $C_1$ and $C_2$ are depending only on $\alpha,\beta,K$ and $\delta$. In the sequel, we will apply Proposition 2.1 to a uniformly $C^{2}$-domain to prove the inequality (1.10).   
\end{rems}

The estimate (2.2) implies the compactness from $\hat{W}^{1,\infty}_{d}(D)$ to $L^{p}(D)$.

\begin{lem}
Let $D$ be a bounded domain in $\mathbf{R}^{n}, n\geq 2$, with $C^{2}$-boundary. Then, there exists a constant $C_{D}$ such that the estimate (2.1) holds for all $\varphi \in \hat{W}^{1,\infty}_{d}(D)$. Moreover, the space $\hat{W}^{1,\infty}_{d}(D)$ is compactly embedded into $L^{p}(D)$. 
 \end{lem}

\begin{proof}
We argue by contradiction. Suppose that the estimate (2.1) were false for any choice of the constant $C$. Then, there would exist a sequence of functions $\{\varphi_{m}\}_{m=1}^{\infty}\subset \hat{W}^{1,\infty}_{d}(D)$ such that 
\begin{equation*}
||\varphi_{m}-(\varphi_{m})||_{L^{p}(D)}>m||\nabla \varphi_m||_{L^{\infty}_{d}(D)},\quad m\in\mathbf{N}.
\end{equation*} 
We may assume $(\varphi_m)=0$ by replacing $\varphi_m$ to $\varphi_m-(\varphi_m)$. We divide $\varphi_m$ by $M_m=||\varphi_m||_{L^{p}(D)}$ to get a sequence of functions $\{\phi_m\}_{m=1}^{\infty}$, $\phi_m=\varphi_{m}/M_m$ such that 
\begin{align*}
&||\nabla \phi_m||_{L^{\infty}_{d}(D)}<1/m,\\
&||\phi_m||_{L^{p}(D)}=1 \quad \textrm{with}\ (\phi_m)=0.
\end{align*}  
We now prove the compactness of $\{\phi_{m}\}_{m=1}^{\infty}$ in $L^{p}(D)$. Since $||\nabla \phi_m||_{L^{\infty}_{d}(D)}$ is bounded, there exists a subsequence still denoted by $\{\phi_m\}_{m=1}^{\infty}$ such that $\phi_{m}$ converges to a limit $\bar\phi$ locally uniformly in $D$. By Proposition 2.1, in particular, the estimate (2.2) implies that $\phi_m$ is uniformly bounded from above by an $L^{p}$-integrable function near $\partial D$. The dominated convergence theorem implies that 
\begin{equation*}
\phi_m\to \bar\phi\quad \textrm{in}\ L^{p}(D)\quad \textrm{as}\ \ m\to\infty .
\end{equation*} 
Since $\nabla \phi_m(x)\to 0$ as $m\to \infty$ for each $x\in D$ and $||\bar\phi||_{L^{p}(D)}=1$, $\bar\phi$ is a non-zero constant which contradicts the fact that $(\bar\phi)=0$. We reached a contradiction.\\
For the compactness of $\{\phi_m\}_{m=1}^{\infty}$ in $L^{p}(D)$, we here only invoke the bound for $||\nabla \phi_{m}||_{L^{\infty}_{d}(D)}$. This means that the embedding from $\hat{W}^{1,\infty}_{d}(D)$ into $L^{p}(D)$ is compact. The proof is now complete.
\end{proof}

\subsection {Estimates near the boundary}

We now prove the inequality (1.10) for uniformly $C^{2}$-domains $\Omega$. When the ball $B_{x_0}(r)$ locates in the interior of $\Omega$, i.e.,  $\Omega_{x_0,r}=B_{x_0}(r)$, applying (2.1) to $\varphi_{r}(x)=\varphi(x_0+rx)$ in $D=B_{0}(1)$ implies the estimate  
\begin{equation*}
||\varphi-(\varphi)||_{L^{p}(\Omega_{x_0,r})}\leq Cr^{n/p}\sup_{z\in \Omega_{x_0,r}}d_{\Omega_{x_0,r}}(z)|\nabla \varphi(z)|,\quad r>0.   \tag{2.3}
\end{equation*}  
Since $d_{\Omega_{x_0,r}}(x)\leq d_{\Omega}(x)$ for $x\in \Omega_{x_0,r}$, the assertion (1.10) follows. However, if $B_{x_0}(r)$ involves $\partial\Omega$, the boundary of $\Omega_{x_0,r}$ may not have $C^{1}$-regularity.
We thus prove
\begin{equation*}
||\varphi-(\varphi)||_{L^{p}(\Omega_{x_0,r})}\leq Cr^{n/p}\sup_{z\in \Omega_{x_0,r}}d_{\Omega}(z)|\nabla \varphi(z)| \quad \textrm{for}\ \varphi\in \hat{W}^{1,\infty}_{d}(\Omega)       \tag{2.4}
\end{equation*}
for $x_0\in \Omega$ and $r>0$ satisfying $d_{\Omega}(x_0)<r$, which is weaker than (2.3).
 
\begin{prop}
Let $\Omega$ be a uniformly $C^{2}$-domain. There exists constants $r_0$ and $C$ such that for $x_0\in \Omega$ and $r\leq r_0$ satisfying $d_{\Omega}(x_0)<r$, the estimate (2.4) holds for all $\varphi\in \hat{W}^{1,\infty}_{d}(\Omega)$ with a constant $C$ independent of $x_0$ and $r$.  
\end{prop}

The inequality (1.10) easily follows from Proposition 2.4.

\begin{lem}
The inequality (1.10) holds for $\varphi\in \hat{W}^{1,\infty}_{d}(\Omega)$ for all $x_{0}\in \Omega$ and $r\leq r_0$ with a constant $C$ independent of $x_{0}$ and $r$. 
\end{lem}

\begin{proof}
For $r<r_0$, combining (2.3) for $d_{\Omega}(x_{0})\geq r$ with (2.4) for $d_{\Omega}(x_{0})<r$, the assertion (1.10) follows. 
\end{proof}

\noindent
\textit{Proof of Proposition 2.4.}
We argue by contradiction. Suppose that the estimate (2.4) were false for any choice of constants $r_0$ and $C$. Then, there would exist a sequence of functions $\{\varphi\}_{m=1}^{\infty}\subset \hat{W}^{1,\infty}_{d}(\Omega)$ and a sequence of points $\{x_m\}_{m=1}^{\infty}\subset \Omega$ satisfying $d_{\Omega}(x_m)<r_m\downarrow 0$ such that 
\begin{equation*}
||\varphi_{m}-(\varphi_{m})||_{L^{p}(\Omega_{x_m,r_m})}>m{r_m}^{n/p}\sup_{z\in \Omega_{x_m,r_m}}d_{\Omega}(z)|\nabla \varphi_m(z)|, \quad m\in\mathbf{N}.
\end{equation*}   
Replacing $\varphi_m$ by $\varphi_m-(\varphi_{m})$ and dividing $\varphi_{m}$ by $r_{m}^{-n/p}||\varphi_m||_{L^{p}(\Omega_{x_m,r_m})}$ (still denoted by $\varphi_m$), we observe that $\varphi_m$ satisfies 
\begin{align*}
&\sup_{z\in \Omega_{x_m,r_m}}d_{\Omega}(z)|\nabla \varphi_m(z)|<1/m,\\
&{r_{m}}^{-n/p}||\varphi_m||_{L^{p}(\Omega_{x_m,r_m})}=1 \quad \textrm{with}\ (\varphi_m)=0.
\end{align*}
Since the points $\{x_m\}_{m=1}^{\infty}$ accumulate at the boundary $\partial\Omega$, we may assume by rotation and translation of $\Omega$ that $x_m=(0,d_{m})$ with $d_m=d_{\Omega}(x_m)$ which subsequently converges to the origin located on the boundary $\partial\Omega$. Here, the neighborhood of the origin is denoted by $\Omega_{\textrm{loc}}=U(0) \cap \Omega$ with constants $\alpha,\beta$ and $C^{2}$-function $h$, i.e.,  
\begin{equation*}
\Omega_{\textrm{loc}}=\{(x',x_n)\in \mathbf{R}^{n}_{+}\ |\ h(x')<x_n<h(x')+\beta,\ |x'|<\alpha \}. 
\end{equation*}   
We rescale ${\varphi}_{m}$ around the point $x_m$ by setting 
\begin{equation*}
\phi_{m}(x)={\varphi}_{m}(x_m+r_mx)\quad \textrm{for}\ x\in \Omega^{m},
\end{equation*}   
where $\Omega^{m}=\{x\in \mathbf{R}^{n}\ |\ x=(y-x_m)/r_m, y\in \Omega \}$ is the rescaled domain.
Since $c_m=d_m/r_m<1$, by taking a subsequence we may assume $\lim_{m\to \infty} c_m=c_0\leq1$. We then observe that the rescaled domain $\Omega^{m}$ expands to a half space $\mathbf{R}_{+,-c_0}^{n}=\{(x',x_n)\in \mathbf{R}^{n}\ |\ x_n>-c_0 \}$. In fact, the neighborhood $\Omega_{\textrm{loc}}\subset \Omega$ is rescaled to the domain,
\begin{equation*}
\Omega^{m}_{\textrm{loc}}=\left\{(x',x_n)\in \mathbf{R}^{n}\ \Bigg|\ \frac{1}{r_m}h(r_m x')-c_m<x_n<\frac{1}{r_m}h(r_m x')+\frac{\beta}{r_m},\ |x'|<\frac{\alpha}{r_m} \right\}
\end{equation*}    
which converges to $\mathbf{R}_{+,-c_0}^{n}$ by letting $m\to \infty$. Note that constants of uniformly regularity of $\partial\Omega_{m}$ are uniformly bounded under this rescaling procedure. Moreover, for any constants $\mu$ and $\alpha'$, the curvilinear neighborhood of the origin ${\cal{N}^{\mu}}(\Gamma)$ is in $\Omega^{m}_{\textrm{loc}}$ for sufficiently large $m\geq 1$, where $\Gamma=\Gamma_{\alpha'}(0)$ is the neighborhood of the origin on $\partial\Omega^{m}$. Then, the estimates for $\varphi_m$ are inherited to the estimates for $\phi_m$, i.e., 
\begin{align*}
&\sup_{z\in \Omega^{m}_{0,1}}d_{\Omega^{m}}(z)|\nabla \phi_m(z)|<1/m,\quad m\in \mathbf{N},\\
&||\phi_{m}||_{L^{p}(\Omega^{m}_{0,1})}=1\quad \textrm{with}\ 
(\phi_m)=\fint_{\Omega^{m}_{0,1}}\phi_{m}=0,
\end{align*}
where $\Omega^{m}_{0,1}=B_{0}(1)\cap\Omega^{m}$. From the above bound for $\nabla \phi_m$, there exists a subsequence still denoted by $\{\phi_{m}\}_{m=1}^{\infty}$ such that $\phi_m$ converges to a limit $\bar\phi$ locally uniformly in $(\mathbf{R}^{n}_{+,-c_0})_{0,1}=  \mathbf{R}^{n}_{+,-c_0}\cap B_{0}(1)$. \\
\noindent
We now observe the compactness of the sequence $\{\phi_m\}_{m=1}^{\infty}$ in $L^{p}((\mathbf{R}^{n}_{+,-c_0})_{0,1})$. By Remark 2.2 (ii), we apply Proposition 2.1 to $\Omega^{m}$ to get the estimate (2.2) with $x_{1}=0,r=1$ and a fixed $y\in \Omega^{m}_{0,1}$ satisfying $d_{\Omega^{m}}(y)\geq \delta$, i.e., 
\begin{equation*}
|\phi_m(x)|\leq C(|\log{d_{\Omega_m}(x)}|+1)\left(\sup_{z\in \Omega^{m}_{0,1}}d_{\Omega_m}(z)|\nabla \phi_m(z)|\right)+|\phi_m(y)|\quad \textrm{for}\ x\in \Omega^{m}_{0,1},    
\end{equation*}
for sufficiently large $m\geq 1$. Here, the constant $C$ is independent of $m\geq 1$. Since $\phi_m$ is uniformly bounded from above by an $L^{p}$-integrable function in $\Omega^{m}_{0,1}$, the dominated convergence theorem implies that $\phi_m$ converges to a limit $\bar\phi$ in $L^{p}((\mathbf{R}^{n}_{+,-c_0})_{0,1})$. Since $\nabla \phi_m(x)\to 0$ as $m\to \infty$ for each $x\in (\mathbf{R}^{n}_{+,-c_0})_{0,1}$ and $||\bar\phi||_{L^{p}((\mathbf{R}^{n}_{+,-c_0})_{0,1})}=1$, $\bar\phi$ is a non-zero constant which contradicts the fact that $(\bar\phi)=0$. We reached a contradiction and the proof is now complete.

\section{A priori estimates for the Stokes equations }

The goal of this section is to prove the a priori estimate (1.4) by using the inequality (1.10). A key step is to establish the estimates for $h$ and $g$ in the procedure (ii) as explained in the introduction. We first recall the $L^{p}$-estimates to the Stokes equations (1.7) and the interpolation inequality (1.11). Note that the constant $C_p$ in (1.7) depends on $r_0$ and $\Omega''$ but independent of parameters $\eta\geq 1$ and $r\leq r_0$ satisfying $(\eta+2)r\leq r_0$.

\subsection{$L^{p}$-estimates for localized equations }

Let $\Omega''$ be a bounded domain with $C^{2}$-boundary. For the a priori estimate (1.4), we invoke the $L^{p}$-estimates (1.7) to the  Stokes resolvent equations with inhomogeneous divergence condition,
\begin{align*}
\lambda u-\Delta u+\nabla p&=h\quad \textrm{in}\ \Omega'',\tag{3.1}\\
\textrm{div}\ u&=g\quad \textrm{in}\ \Omega'',\tag{3.2}\\
 u&=0\quad \textrm{on}\ \partial\Omega'',\tag{3.3}
\end{align*}  
for $h\in L^{p}(\Omega''), g\in W^{1,p}(\Omega'')\cap L^{p}_{\textrm{av}}(\Omega'')$ and $\lambda\in \sum_{\vartheta,0}$ where $\vartheta\in (\pi/2,\pi)$. Here, $L^{p}_{\textrm{av}}(\Omega'')$ denotes the space of all functions $g$ in $L^{p}(\Omega'')$ satisfying average zero, i.e., $\int_{\Omega''}g d x=0$. 

\begin{prop}{\normalfont{(\cite{FS92}, \cite[Theorem 1.2]{FS1})}}
Let $\vartheta\in (\pi/2,\pi)$ and $\lambda\in \sum_{\vartheta,0}$. For $h\in L^{p}(\Omega'')$ and $g\in W^{1,p}(\Omega'')\cap L^{p}_{\textrm{av}}(\Omega'')$, there exists a unique solution of (3.1)--(3.3) satisfying the estimate (1.7) with the constant $C_p$ depending on $\vartheta, p ,n$ and the $C^{2}$-regularity of $\partial\Omega''$.
\end{prop}     

We estimate the $L^{\infty}$-norms of a solution up to first derivatives via the Sobolev embeddings together with the $L^{p}$-estimates (1.7) for $p>n$. In order to estimate the $L^{\infty}$-norms of a solution, we apply the interpolation inequality (1.11). Actually, if $\Omega_{x_0,r}=B_{x_0}(r)$, the stronger estimate (A.1) holds, i.e., we are able to replace the right-hand side of (1.11) by the norms for $\varphi$ and $\nabla\varphi$ on $B_{x_0}(r)$. However, if $B_{x_0}(r)$ is near the boundary $\partial\Omega$, $\partial\Omega_{x_0,r}$ may not be $C^{1}$-boundary. We thus estimate the sup-norm of $\varphi$ in $\Omega_{x_0,r}$ by the norms for $\varphi$ and $\nabla\varphi$ in $\Omega_{x_0,2r}$. In Appendix A, we prove the inequality (1.11) with the constant $C_{I}$ independent of $x_0$ and $r$; see Lemma A.2. In what follows, we fix the constant $r_0$ with the same constant $r_0$ given in Lemma 2.5.

\subsection{Estimates in the localization procedure }

We prepare the estimates for $h$ and $g$ in the procedure (ii). The estimate for $|\lambda|||g||_{W^{-1,p}_{0}}$ is different from that of $||h||_{L^{p}}$. In order to estimate $|\lambda|||g||_{W^{-1,p}_{0}}$, we use the uniformly local $L^{p}$-norm bound for $\nabla q$ besides the sup-bound of $\nabla v$ as in (3.7). After establishing these estimates, we will put the procedures (i)-(iii) together in the next subsection.\\     

Let $\Omega$ be a uniformly $C^{2}$-domain. Let $\theta$ be a smooth cut-off function satisfying $\theta\equiv1$ in $[0,1/2]$ and $\theta\equiv 0$ in $[1,\infty)$. For $x_0\in \Omega$ and $r>0$, we set $\theta_{0}(x)=\theta(|x-x_0|/(\eta+1)r)$ with parameters $\eta\geq 1$ and observe that $\theta_{0}\equiv1$ in $B_{x_0}(r)$ and $\theta_{0}\equiv 0$ in $B_{x_0}((\eta+1)r)^{c}$. The cut-off function $\theta_{0}$ is uniformly bounded by a constant $K$, i.e.,  
\begin{equation*}
||\theta_{0}||_{\infty}+(\eta+1)r||\nabla \theta_{0}||_{\infty}+(\eta+1)^{2}r^{2}||\nabla^2\theta_0||_{\infty}\leq K,\quad\textrm{for}\ \eta\geq 1.  \tag{3.4}
\end{equation*}       
Let $(v,\nabla q)\in W^{2,p}_{\textrm{loc}}(\bar{\Omega})\times L^{p}_{\textrm{loc}}(\bar{\Omega})$ be a solution of (1.1)--(1.3) for $f\in L^{\infty}_{\sigma}(\Omega)$ and $\lambda \in \Sigma_{\vartheta,0}$. We localize a solution $(v,\nabla q)$ in the domain $\Omega'=\Omega_{x_{0},(\eta+1)r}$ by setting $u=v\theta_{0}$ and $p=\hat{q}\theta_{0}$ where $\hat{q}=q-q_c$ and a constant $q_{c}$. Then, $(u,\nabla p)$ solves the localized equation (3.1)--(3.3) in the domain $\Omega'$ with $h$ and $g$ given by (1.8). We take parameters $\eta\geq 1$ and $r>0$ such that $(\eta+2)r\leq r_0$. Since we adjust parameters $\eta\geq 1$ later, we take a $C^{2}$-bounded domain $\Omega''$ such that $\Omega_{x_0,r_0}\subset \Omega''$ and apply the $L^{p}$-estimate (1.7) in $\Omega''$. Note that $\Omega'\subset \Omega''$ for all $\eta\geq 1$ and $r>0$ satisfying $(\eta+2)r\leq r_0$. We shall  show the following estimates for $h$ and $g$:  
\begin{align*}
||\nabla g||_{L^{p}(\Omega'')}
&\leq C_1r^{n/p}(\eta+1)^{-(1-n/p)}\left(r^{-1}|| \nabla v||_{L^{\infty}(\Omega)}+r^{-2}||v||_{L^{\infty}(\Omega)}\right),    \tag{3.5} \\  
||h||_{L^{p}(\Omega'')}
&\leq C_2r^{n/p}\Bigg((\eta+1)^{n/p}||f||_{L^{\infty}(\Omega)}\\
&+(\eta+1)^{-(1-n/p)}\Big(r^{-1}|| \nabla v||_{L^{\infty}(\Omega)}+r^{-2}||v||_{L^{\infty}(\Omega)}\Big)\Bigg),   \tag{3.6}  \\
|\lambda|||g||_{W^{-1,p}_{0}(\Omega'')}
&\leq C_3r^{n/p}\Bigg((\eta+1)^{n/p}||f||_{L^{\infty}(\Omega)}\\
&+(\eta+1)^{-(1-2n/p)}\Big(r^{-1}|| \nabla v||_{L^{\infty}(\Omega)}
+r^{-n/p}\sup_{z\in\Omega}||\nabla q||_{L^{p}(\Omega_{z,r})}\Big)\Bigg).    \tag{3.7}
\end{align*}   
The constants $C_1, C_2$ and $C_3$ are independent of $r$ and $\eta\geq 1$ satisfying $(\eta+2)r\leq r_0$. Since $h$ and $g$ are supported in $\Omega'$, we have $||h||_{L^{p}(\Omega')}=||h||_{L^{p}(\Omega'')}$ and $||\nabla g||_{L^{p}(\Omega')}=||\nabla g||_{L^{p}(\Omega'')}$.

For the estimates of the terms $f, v$ and $\nabla v$, we use the estimates 
\begin{align*}   
||f\theta_{0}||_{L^{p}(\Omega')}
&\leq KC_{n}^{1/p}r^{n/p}(\eta+1)^{n/p}||f||_{L^{\infty}(\Omega)},\tag{3.8}\\
||\nabla v\nabla \theta_{0}||_{L^{p}(\Omega')}
&\leq KC_{n}^{1/p}r^{n/p}(\eta+1)^{-(1-n/p)}r^{-1}||\nabla v||_{L^{\infty}(\Omega)},   \tag{3.9} \\
||v\nabla^{2} \theta_{0}||_{L^{p}(\Omega')}
&\leq KC_{n}^{1/p}r^{n/p}(\eta+1)^{-(1-n/p)}r^{-2}|| v||_{L^{\infty}(\Omega)},   \tag{3.10}
\end{align*}  
for all $r>0$ and $\eta\geq 1$, where the constant $C_{n}$ denotes the volume of the $n$-dimensional unit ball. Since $\nabla g=\nabla v\nabla \theta_{0}+v\nabla^{2}\theta_{0}$ does not contain the pressure, the estimate (3.5) easily follows from the estimates (3.9) and (3.10). 

For the estimates (3.6) and (3.7), we apply the inequality (1.10). We choose a constant $q_{c}$ by a mean value of $q$ in $\Omega_{x_{0},(\eta+2)r}$, i.e.,  
\begin{equation*}
q_{c}=\fint_{\Omega_{x_0,(\eta+2)r}}q(x)dx.   \tag{3.11}
\end{equation*} 
We then observe that the inequality (1.10) implies the estimate  
\begin{equation*}
||\hat{q}||_{L^{p}(\Omega_{x_0,(\eta+2)r})}\leq Cr^{n/p}(\eta+2)^{n/p}||\nabla q||_{L^{\infty}_{d}(\Omega)}   \tag{3.12}
\end{equation*}
for $r>0$ and $\eta\geq 1$ satisfying $(\eta+2)r\leq r_0$, where $\hat{q}=q-q_c$.

In order to show the estimate (3.7), we estimate the $L^{\infty}$-norm of $\hat{q}$ on $\Omega'$ since by using the equation $\lambda v=f+\Delta v-\nabla q$, we reduce (3.7) to the estimate of the boundary value of $\hat{q}$ on $\partial\Omega'$. 
This is the reason why we take $q_c$ by (3.11). We apply the inequality (1.11) in $\Omega_{x_1,r/2}\subset \Omega_{x_{0},(\eta+2)r}$ for $x_1\in \Omega'$ and $r\leq r_0$ with $p>n$ to estimate
\begin{align*}
||\hat{q}||_{L^{\infty}(\Omega_{x_1,r/2})}
&\leq C_{I}r^{-n/p}\Big(||\hat{q}||_{L^{p}(\Omega_{x_1,r})}+r||\nabla q||_{L^{p}(\Omega_{x_1,r})} \Big)\\
&\leq C_{I}r^{-n/p}\Big(||\hat{q}||_{L^{p}(\Omega_{x_0,(\eta+2)r})}+r\sup_{z\in \Omega}||\nabla q||_{L^{p}(\Omega_{z,r})} \Big).    \tag{3.13}
\end{align*}
Combining the estimate (3.13) with (3.12) and taking a supremum for $x_{1}\in \Omega'$, we have
\begin{equation*}
||\hat{q}||_{L^{\infty}(\Omega')}
\leq C\Big((\eta+2)^{n/p}||\nabla q||_{L^{\infty}_{d}(\Omega)}+r^{1-n/p}\sup_{z\in \Omega}||\nabla q||_{L^{p}(\Omega_{z,r})}\Big).   \tag{3.14}
\end{equation*}
We now invoke the strictly admissibility of a domain $\Omega$ to estimate the norm $||\nabla q||_{L^{\infty}_{d}(\Omega)}$ by the sup-norm of $\nabla v$ in $\Omega$ via (1.5).   

\begin{prop}
Let $\Omega$ be a uniformly $C^{2}$-domain. Assume that $\Omega$ is strictly admissible. Then, the estimate
\begin{equation*}     
||\hat{q}||_{L^{p}(\Omega')}\leq C_{4}r^{n/p}(\eta+2)^{n/p}||\nabla v||_{L^{\infty}(\Omega)}  \tag{3.15}
\end{equation*}
holds for all $r>0$ and $\eta\geq 1$ satisfying $(\eta+2)r\leq r_{0}$ and $p\in [1,\infty)$. If in addition $p>n$, then the estimate 
\begin{equation*}
||\hat{q}||_{L^{\infty}(\Omega')}
\leq C_5
\left(
(\eta+2)^{n/p}||\nabla v||_{L^{\infty}(\Omega)}
    +r^{1-n/p}\sup_{z\in \Omega}||\nabla q||_{L^{p}(\Omega_{z,r})}
    \right)   \tag{3.16}
\end{equation*}
 holds. The constants $C_4$ and $C_5$ are independent of $r$ and $\eta$. 
\end{prop} 

\begin{proof}
By (1.5), (3.12) and (3.14), the assertion follows. 
\end{proof}

By using the estimates (3.15) and (3.16), we obtain the estimates (3.6) and (3.7).
\begin{lem}
Let $\Omega$ be a strictly admissible, uniformly $C^{2}$-domain. Let $(v,\nabla q)\in W^{2,p}_{\textrm{loc}}(\bar\Omega)\times (L^{p}_{\textrm{loc}}(\bar\Omega)\cap L^{\infty}_{d}(\Omega))$ be a solution of (1.1)--(1.3) for $f\in L^{\infty}_{\sigma}(\Omega)$, $\lambda\in \sum_{\vartheta,0}$ and  $p>n$. Then, the estimates (3.5)--(3.7) hold for $\Omega'=B_{x_0}((\eta+1)r)\cap \Omega$, $x_0\in \Omega$,  $r>0$ and $\eta\geq 1$ satisfying $(\eta+2)r\leq r_{0}$ with the constants $C_1$, $C_2$ and $C_3$ independent of $x_0$, $r$ and $\eta$.
\end{lem}

\begin{proof}
As mentioned before, (3.5) follows from (3.9) and (3.10). The estimate (3.6) follows from the estimates (3.8)--(3.10) and (3.15). We shall show the estimate (3.7). Since $||g||_{W^{-1,p}_{0}(\Omega'')}\leq ||g||_{W^{-1,p}_{0}(\Omega')}$, we estimate $||g||_{W^{-1,p}_{0}(\Omega')}$. Note that $\partial\Omega'$ may not be $C^{1}$ on the intersection $\partial\Omega\cap B_{x_0}((\eta+1)r)$. We first show (3.7) with assuming that $\partial\Omega'$ has $C^{1}$-boundary. By using the equation $\lambda g=\lambda v\cdot\nabla \theta_{0}=(f+\Delta v-\nabla q)\cdot\nabla \theta_{0}$, we estimate
\begin{equation*}
|\lambda|||g||_{W^{-1,p}_{0}(\Omega')}\leq ||f\cdot\nabla\theta_{0}||_{W^{-1,p}_{0}(\Omega')}+||\Delta v\cdot\nabla \theta_{0}||_{W^{-1,p}_{0}(\Omega')}+||\nabla q\cdot\nabla \theta_0||_{W^{-1,p}_{0}(\Omega')}.
\end{equation*}
Since $||f\cdot\nabla\theta_{0}||_{W^{-1,p}_{0}(\Omega')}\leq||f\theta_{0}||_{L^{p}(\Omega')}$ for $f\in L^{\infty}_{\sigma}(\Omega)$,
it suffices to show the estimates 
 \begin{align*}
 ||\Delta v\cdot\nabla \theta_{0}||_{W^{-1,p}_{0}(\Omega')}
 &\leq C_{6}r^{n/p}(\eta+1)^{-(1-n/p)}r^{-1}||\nabla v||_{L^{\infty}(\Omega)},          \tag{3.17}\\
 ||\nabla q\cdot\nabla \theta_0||_{W^{-1,p}_{0}(\Omega')}
 &\leq C_{7}r^{n/p}(\eta+1)^{-(1-2n/p)}\Bigg(r^{-1}||\nabla v||_{L^{\infty}(\Omega)}+r^{-n/p}\sup_{z\in \Omega}||\nabla q||_{L^{p}(\Omega_{z,r})}\Bigg).  \tag{3.18}
 \end{align*}
We first show (3.17). Take $\varphi\in W^{1,p'}(\Omega')$ satisfying $||\varphi||_{W^{1,p'}(\Omega')}\leq 1$. By using $\textrm{div}\ v=0$, integration by parts yields that
 \begin{equation*}
 \sum_{i,j=1}^{n}\int_{\Omega'}\partial_{j}^{2}v^{i}\partial_{i}\theta_{0}\varphi dx
 =\sum_{i,j=1}^{n}\int_{\Omega'}(\partial_{j}v^{i}-\partial_{i}v^{j})\partial_{j}\theta_{0}\partial_{i}\varphi dx
 -\int_{\partial\Omega'}(\partial_{j}v^{i}-\partial_{i}v^{j})\partial_{j}\theta_{0}\varphi n^{i}_{\Omega}d{\cal{H}}^{n-1}(x).  
\end{equation*}  
We estimate the second term in the right-hand side by the  $W^{1,1}$-norm of $\varphi$ in $\Omega'$ \cite[5.5 Theorem 1.1]{E} to estimate
\begin{equation*}
||\varphi||_{L^{1}(\partial\Omega')}
\leq C_{T}||\varphi||_{W^{1,1}(\Omega')} \notag  
\leq 2C_{T}|\Omega'|^{1/p}    \tag{3.19}
\end{equation*}
with the constant $C_{T}$ depending on the $C^{1}$-regularity of the boundary $\partial\Omega$ but independent of $|\Omega'|$, the volume of $\Omega'$. We thus obtain 
\begin{equation*}
\begin{split}
\left|\sum_{i,j=1}^{n}\int_{\Omega'}\partial^{2}_{j}v^{i}\partial_{i}\theta_{0}\varphi dx\right|
&\leq (1+2C_{T})\sum_{i,j=1}^{n}||(\partial_{j}v^{i}-\partial_{i}v^{j})\partial_{j}\theta_{0}||_{L^{\infty}(\Omega')} |\Omega'|^{1/p}\\
&\leq 4n(1+2C_{T} )K {C_{n}}^{1/p}r^{n/p}(\eta+1)^{-(1-n/p)}r^{-1}||\nabla v||_{L^{\infty}(\Omega)}.
\end{split}
\end{equation*} 
Thus, the estimate (3.17) holds with the constant $C_6$ independent of $r$ and $\eta$. It remains to show the estimate (3.18). Since $\nabla q=\nabla \hat{q}$, integration by parts yields that 
\begin{equation*}
\begin{split} 
\int_{\Omega'}\nabla q\cdot\nabla \theta_{0}\varphi dx
&=-\int_{\Omega'} \hat{q}(\Delta\theta_{0}\varphi+\nabla \theta_{0}\cdot \nabla \varphi)dx+\int_{\partial\Omega'}\hat{q}\varphi\nabla\theta_{0}\cdot n_{\Omega'}d{\cal{H}}^{n-1}(x)\\
&= I+II+III.
\end{split}
\end{equation*} 
Combining (3.4), (3.19) with (3.16), we obtain
\begin{align*}
II +III
&\leq (1+2C_{T})||\hat{q}\nabla \theta_{0}||_{L^{\infty}(\Omega')}|\Omega'|^{1/p}\\
&\leq (1+2C_{T})K{C_{n}}^{1/p}r^{n/p}(\eta+1)^{-(1-n/p)}r^{-1}||\hat{q}||_{L^{\infty}(\Omega')}\\
&\leq Cr^{n/p}(\eta+1)^{-(1-2n/p)}\left(r^{-1}||\nabla v||_{L^{\infty}(\Omega)}+r^{-n/p}\sup_{z\in\Omega}||\nabla q||_{L^{p}(\Omega_{z,r})} \right),     
\end{align*}
with the constant $C$ depending on $C_T, K, C_n, p, C_4$ and $C_5$ but independent of $r$ and $\eta$. We complete the proof by showing the estimate for $I$. Applying the H\"{o}lder inequality, for $s,s'\in (1,\infty)$ with $1/s+1/{s'}=1$ we have 
\begin{equation*}
I\leq K(\eta+1)^{-2}r^{-2}||\varphi||_{L^{s}(\Omega')}||\hat{q}||_{L^{s'}(\Omega')}.         
\end{equation*}  
Since $p>n$, the conjugate exponent $p'$ is strictly smaller than $n/(n-1)$ for $n\geq 2$. By setting $1/s=1/{p'}-1/n$, we apply the Sobolev inequality \cite[5.6 Theorem 2]{E} to estimate $||\varphi||_{L^{s}(\Omega')}  
\leq C_{S}||\varphi||_{W^{1,p'}(\Omega')} 
\leq C_{S}$ with the constant $C_s$ independent of $|\Omega'|$. Applying the estimate (3.15) to $\hat{q}$ yields 
\begin{align*}
I
&\leq Cr^{n/s'-2}(\eta+2)^{n/s'-2}||\nabla v||_{L^{\infty}(\Omega)}\\   
&\leq Cr^{n/p}(\eta+2)^{-(1-n/p)}r^{-1}||\nabla v||_{L^{\infty}(\Omega)},
\end{align*}
since $1/s'=1-1/s=1/p+1/n$. The constant $C$ is independent of $r$ and $\eta$. Thus, we proved (3.7) with assuming the $C^{1}$-regularity for $\partial\Omega'$.

If $\partial\Omega'$ is not $C^{1}$, we modify $\Omega'$ around the intersection $\partial\Omega\cap B_{x_0}((\eta+1)r)$, i.e., we take a $C^{1}$-bounded domain $\tilde{\Omega}'\subset\Omega''$ such that $\Omega'\subset \tilde{\Omega}'$ and $|\tilde{\Omega}'|\leq C|\Omega'|$ with the constant $C$ depending on the $C^{1}$-regularity of $\partial\Omega$, but independent of $|\Omega'|$. For example, we take a $C^{1}$-domain $\tilde{\Omega}'$ such that $\Omega_{x_0,(\eta+1)r}\subset \tilde{\Omega}'\subset \Omega_{x_0,(\eta+3/2)r}$. Since $\tilde{\Omega}'\subset \Omega''$ and $g$ is supported in $\Omega'$, it follows that $||g||_{W^{-1,p}_{0}(\Omega'')}\leq ||g||_{W^{-1,p}_{0}(\tilde{\Omega}')}$. Then, we are able to estimate $||g||_{W^{-1,p}_{0}(\tilde{\Omega}')}$ in the same way as above. In fact, we are able to show the estimates: 
 \begin{align*}
 ||\Delta v\cdot\nabla \theta_{0}||_{W^{-1,p}_{0}(\tilde{\Omega}')}
 &\leq C_{6}'r^{n/p}(\eta+1)^{-(1-n/p)}r^{-1}||\nabla v||_{L^{\infty}(\Omega)},          \tag{3.17'}\\
 ||\nabla q\cdot\nabla \theta_0||_{W^{-1,p}_{0}(\tilde{\Omega}')}
 &\leq C_{7}'r^{n/p}(\eta+1)^{-(1-2n/p)}\Bigg(r^{-1}||\nabla v||_{L^{\infty}(\Omega)}+r^{-n/p}\sup_{z\in \Omega}||\nabla q||_{L^{p}(\Omega_{z,r})}\Bigg).  \tag{3.18'}
 \end{align*}
The estimates (3.7) follows from (3.17') and (3.18'). The estimate (3.17') follows by the same way with (3.17) since $\partial\Omega'$ is $C^{1}$ and $|\tilde{\Omega}'|\leq C|\Omega'|$.

We shall show (3.18'). Since $||\hat{q}||_{L^{\infty}(\tilde{\Omega}')}\leq ||\hat{q}||_{L^{\infty}(\Omega_{x_0,(\eta+2)r})}=\sup_{x_1\in \Omega'}||\hat{q}||_{L^{\infty}(\Omega_{x_1,r/2})}$, the stronger estimate than (3.14) holds, i.e., 
\begin{equation*}
||\hat{q}||_{L^{\infty}(\tilde{\Omega}')}
\leq C\Big((\eta+2)^{n/p}||\nabla q||_{L^{\infty}_{d}(\Omega)}+r^{1-n/p}\sup_{z\in \Omega}||\nabla q||_{L^{p}(\Omega_{z,r})}\Big).  
\end{equation*}
Thus, we are able to replace the left-hand side of (3.15) and (3.16) by $||\hat{q}||_{L^{\infty}(\tilde{\Omega}')}$. Then, the estimate (3.18') follows by the same way with (3.18). 

We proved (3.7). The proof is now complete.  
\end{proof}

\begin{rem}\normalfont
From the estimate (3.7), we observe that the exponent $-(1-2n/p)$ of $(\eta+1)$ in front of the term $(r^{-1}||\nabla v||_{L^{\infty}(\Omega)}+r^{-n/p}\sup_{z\in \Omega}||\nabla q||_{L^{p}(\Omega_{z,r})})$ is negative provided that $p>2n$. We thus first prove the a priori estimate (1.4) for $p>2n$. Once we obtain the estimate $|\lambda|||v||_{L^{\infty}(\Omega)}\leq C||f||_{L^{\infty}(\Omega)}$, it is easy to replace the estimate (3.7) to 
\begin{equation*}
|\lambda|||g||_{W^{-1,p}_{0}(\Omega')}\leq C K {C_{n}}^{1/n} r^{n/p} (\eta+1)^{n/p}||f||_{L^{\infty}(\Omega)}    
\end{equation*}
for $p>n$ since 
\begin{equation*}
\begin{split}
|\lambda|||v\cdot\nabla\theta_0||_{W^{-1,p}_{0}(\Omega')}
&\leq|\lambda|||v\theta_0||_{L^{p}(\Omega)}\\
&\leq C||\theta_0||_{L^{p}(\Omega')}||f||_{L^{\infty}(\Omega)}\\
&\leq C K {C_{n}}^{1/p}r^{n/p}(\eta+1)^{n/p}||f||_{L^{\infty}(\Omega)}.
\end{split}
\end{equation*}
\end{rem}

\subsection{Interpolation}

We now prove the a priori estimate (1.4) for $p>n$. The parameters $\eta$ and the constant $\delta$ are determined only through the constants $C_{p}, C_{I}$ and $C_{1}$--$C_{3}$. Although we eventually obtain the estimate (1.12) for all $p>n$, firstly we prove the case $p>2n$ as observed by Remark 3.4. The case $p>2n$ is enough for analyticity but, for the completeness, we prove the estimate (1.4) for all $p>n$.\\  

\noindent
\textit{Proof of Theorem 1.1.}  
We set $\delta=\delta_{\eta}=(\eta+2)^{2}/{r_0}^{2}$ and now take $r=1/|\lambda|^{1/2}$ for $\lambda\in \sum_{\vartheta,\delta}$. We then observe that $r=1/|\lambda|^{1/2}$ and $\eta\geq 1$ automatically satisfy $r(\eta+2)\leq r_0$ for $\lambda\in \Sigma_{\vartheta,\delta}$. We take a $C^{2}$-bounded domain $\Omega''$ such that $\Omega_{x_0,r_0}\subset \Omega''\subset \Omega$. Then, $\Omega'\subset \Omega''$ for all $\eta\geq 1$ and $r>0$ satisfying $(\eta+2)r\leq r_0$. We first prove: \\

\noindent
\textit{Case} (I) $p>2n$. We apply the $L^{p}$-estimates (1.7) to $u=v\theta_{0}$ and $p=\hat{q}\theta_{0}$ in $\Omega''$ to get 
\begin{align*}
&|\lambda| ||u||_{L^{p}(\Omega'')}+|\lambda|^{1/2} ||\nabla u||_{L^{p}(\Omega'')}+||\nabla^{2}u||_{L^{p}(\Omega'')}+||\nabla p||_{L^{p}(\Omega'')}\\
&\leq C_{p}\left(||h||_{L^{p}(\Omega'')}+||\nabla g||_{L^{p}(\Omega'')}+|\lambda| ||g||_{W^{-1,p}_{0}(\Omega'')}\right),  
\end{align*} 
where the constant $C_{p}$ depends on $r_0$, but independent of $\eta\geq 1$ and $r>0$ satisfying $(\eta+2)r\leq r_0$. Combining the above estimate and (3.5)--(3.7), we obtain
\begin{align*}
&|\lambda|||u||_{L^{p}(\Omega'')}+|\lambda|^{1/2}||\nabla u||_{L^{p}(\Omega'')}+||\nabla^{2} u||_{L^{p}(\Omega'')}+||\nabla p||_{L^{p}(\Omega'')}\\
&\leq C_8|\lambda|^{-n/2p}\left( (\eta+1)^{n/p}||f||_{L^{\infty}(\Omega)}+(\eta+1)^{-(1-2n/p)}||M_{p}(v,q)||_{L^{\infty}(\Omega)}(\lambda) \right),  \tag{3.20}
\end{align*}    
with the constant $C_{8}$ independent of $r=1/|\lambda|^{1/2}$ and $\eta\geq 1$. We next estimate the $L^{\infty}$-norms of $u$ and $\nabla u$ in $\Omega$ by interpolation. Applying the interpolation inequality (1.11) for $\varphi=u$ and $\nabla u$ implies the estimates  
\begin{align*}
||u||_{L^{\infty}(\Omega_{x_0,r})}
&\leq C_{I}r^{-n/p}\Big(||u||_{L^{p}(\Omega_{x_0,2r})}+r||\nabla u||_{L^{p}(\Omega_{x_0,2r})}\Big),\\
||\nabla u||_{L^{\infty}(\Omega_{x_0,r})}
&\leq C_{I}r^{-n/p}\Big(||\nabla u||_{L^{p}(\Omega_{x_0,2r})}+r||\nabla^{2} u||_{L^{p}(\Omega_{x_0,2r})}\Big).
\end{align*}
Summing up these norms together with $|\lambda|^{n/2p}||\nabla^{2}u||_{L^{p}(\Omega_{x_{0},r})}$ and $|\lambda|^{n/2p}||\nabla p||_{L^{p}(\Omega_{x_{0},r})}$, we have
\begin{align*}
&M_{p}(u,p)(x_0,\lambda)\\
&\leq C_9r^{-n/p}\left(|\lambda|||u||_{L^{p}(\Omega_{x_0,2r})}+|\lambda|^{1/2}||\nabla u||_{L^{p}(\Omega_{x_0,2r})}+||\nabla^{2} u||_{L^{p}(\Omega_{x_0,2r})}+||\nabla p||_{L^{p}(\Omega_{x_0,2r})} \right)         \tag{3.21}
\end{align*} 
with the constant $C_9$ independent of $r>0$ and $\eta\geq 1$. Since $(u,\nabla p)$ agrees with $(v,\nabla q)$ in $\Omega_{x_0,r}$ and $\Omega_{x_0,2r}\subset \Omega''$, combining (3.20) with (3.21) yields  
\begin{equation*}
M_{p}(v,q)(x_0,\lambda)\leq C_{10}\left((\eta+1)^{n/p}||f||_{L^{\infty}(\Omega)}+(\eta+1)^{-(1-2n/p)}||M_{p}(v,q)||_{L^{\infty}(\Omega)}(\lambda)\right)  \tag{3.22}
\end{equation*}  
with $C_{10}=C_{8}C_{9}.$ We take a supremum for $x_0\in \Omega$ and now fix the parameters $\eta\geq 1$ so that $C_{10}(\eta+1)^{-(1-2n/p)}<1/2$. Then, we obtain (1.4) with $C=2C_{10}$ for $p>2n$.\\ 

We shall complete the proof by showing the uniformly local $L^{p}$-bound for second derivatives of $(v,q)$ for all $p>n$.\\

\noindent
\textit{Case} (II) $p>n$. Since $|\lambda|||g||_{W^{-1,\tilde{p}}_{0}}$ is bounded for $\tilde{p}>2n$, we may assume $(v,\nabla q)\in W^{2,\tilde{p}}_{\textrm{loc}}(\bar{\Omega})\times L^{\tilde{p}}_{\textrm{loc}}(\bar{\Omega})$ for $\tilde{p}>2n$. By using $|\lambda|||v||_{L^{\infty}(\Omega)}\leq C||f||_{L^{\infty}(\Omega)}$ for $\lambda \in \Sigma_{\vartheta,\delta}$ with $\delta =\delta_{\tilde{p}}$ we replace the estimate (3.7) to 
\begin{equation*}
|\lambda|||g||_{W^{-1,p}_{0}(\Omega')}\leq CK {C_n}^{1/p}r^{n/p}(\eta+1)^{n/p}||f||_{L^{\infty}(\Omega)}
\end{equation*}
by Remark 3.4. Then, we are able to replace the estimate (3.22) to 
\begin{equation*}
||M_{p}(v,q)||_{L^{\infty}(\Omega)}(\lambda)
\leq C_{11}\left((\eta+1)^{n/p}||f||_{L^{\infty}(\Omega)}+(\eta+1)^{-(1-n/p)}||M_{p}(v,q)||_{L^{\infty}(\Omega)}(\lambda) \right).
\end{equation*} 
Letting $\eta\geq 1$ large so that
$C_{11}(\eta+1)^{-(1-n/p)}<1/2$, we obtain (1.4) for all $p>n$. The proof is now complete.

\begin{rem}\normalfont (Robin boundary condition)
Concerning the Robin boundary condition, we replace the Dirichlet boundary condition for the localized equations (3.3) to the inhomogeneous boundary condition with a tangential vector field $k$,
\begin{equation*}
B(u)=k,\quad u\cdot n_{\Omega''}=0\quad \textrm{on}\ \partial\Omega''.
\end{equation*}  
Instead of the estimate (1.7), we apply the $L^{p}$-estimate of the form,
\begin{align*}
&|\lambda| ||u||_{L^{p}(\Omega'')}+|\lambda|^{1/2} ||\nabla u||_{L^{p}(\Omega'')}+||\nabla^{2}u||_{L^{p}(\Omega'')}+||\nabla p||_{L^{p}(\Omega'')}\\
&\leq C(||h||_{L^{p}(\Omega'')}+||\nabla g||_{L^{p}(\Omega'')}+|\lambda|||g||_{W^{-1,p}_{0}(\Omega'')}+|\lambda|^{1/2}||k||_{L^{p}(\Omega'')}+||\nabla k||_{L^{p}(\Omega'')} ),
\end{align*}      
where $k$ is identified with its arbitrary extension to $\Omega''$. Since $k=v_{\textrm{tan}}\partial \theta_{0}/\partial n_{\Omega''}$ for $u=v\theta_{0}$ and $p=\hat{q}\theta_0$, we observe that the norms of $k$ in the right-hand side are estimated by the same way with $||\nabla g||_{L^{p}}$ where $g=v\cdot \nabla \theta_{0}$. The above $L^{p}$-estimate for the Robin boundary condition is proved by \cite{ShbS} for bounded and exterior domains by generalizing the perturbation argument to the Dirichlet boundary condition \cite{FS1}. After proving the a priori estimate (1.4) for $f\in L^{\infty}_{\sigma}$ subject to the Robin boundary condition, we verify the existence of solutions for (1.1) and (1.2). In particular, $v\in L^{\infty}_{\sigma}$ (not in $C_{0,\sigma}$).  Then, we are able to define the Stokes operator $A=A_{R}$ in $L^{\infty}_{\sigma}$ in the same way as we did for the Dirichlet boundary condition. Our observations may be summarized as following: 
\end{rem}

\begin{thm}
Assume that $\Omega$ is a bounded or an exterior domain with $C^{3}$-boundary in $\mathbf{R}^{n}$. Then, the Stokes operator $A=A_{R}$ subject to the Robin boundary condition generates an analytic semigroup on $L^{\infty}_{\sigma}(\Omega)$ of angle $\pi/2$.  
\end{thm}

\section*{acknowledgements}
The authors are also grateful to the anonymous referees for their valuable comments. The work of first author is supported by Grant-in-aid for Scientific Research of JSPS Fellow No. 24--8019. The work of second author is partially supported by Grant-in-aid for Scientific Research, No. 21224001 (Kiban S), No. 23244015 (Kiban A), No. 20654017 (Houga), the Japan Society for the Promotion of Science (JSPS). This work is supported in part by 
the DFG-JSPS  International Research Training Group 1529 on Mathematical Fluid Dynamics.

\appendix

\section{An interpolation inequality near the boundary}

In Appendix A, we give a proof for the inequality (1.11). The inequality (1.11) holds for all $x_0\in \Omega$ and $r\leq r_0$ in a uniformly $C^{1}$-domain even if $\partial\Omega_{x_0,r}$ is not $C^{1}$. \\

We prove (1.11) for $x_0\in \Omega$ and $r\leq r_0$ by a blow-up argument as we did the inequality (2.4). If $B_{x_0}(r)$ is in the interior of $\Omega$, i.e., $\Omega_{x_0,r}=B_{x_0}(r)$, the inequality (1.11) follows from the Sobolev inequality in $B_{0}(1)$. In fact, applying the Sobolev inequality for $\varphi_{r}(x)=\varphi(x_0+rx)$, $\varphi\in W^{1,p}_{\textrm{loc}}(\bar{\Omega})$ yields
\begin{equation*}
||\varphi_{r}||_{L^{\infty}(B_{0}(1))}\leq C_{s}||\varphi_{r}||_{W^{1,p}(B_{0}(1))}.
\end{equation*}  
Since $||\varphi_{r}||_{L^{p}(B_{0}(1))}=r^{-n/p}||\varphi||_{L^{p}(B_{x_0}(r))}$ and $||\nabla \varphi_{r}||_{L^{p}(B_{0}(1))}=r^{1-n/p}||\nabla \varphi||_{L^{p}(B_{x_0}(r))}$, we have
\begin{equation*}
||\varphi||_{L^{\infty}(B_{x_0}(r))}\leq C_{s}r^{-n/p}\left(||\varphi||_{L^{p}(B_{x_0}(r))}+ r||\nabla \varphi||_{L^{p}(B_{x_0}(r))}\right)    \tag{A.1}
\end{equation*}
for $x_0\in \Omega$ and $r>0$ satisfying $d_{\Omega}(x_0)\geq r$. The inequality (A.1) is stronger than (1.11).

If $B_{x_0}(r)$ is located near the boundary, i.e., $d_{\Omega}(x_0)< r$, $\partial \Omega_{x_0,r}$ may not be $C^{1}$. However, the weaker inequality (1.11) holds since we take the norms on $\Omega_{x_0,2r}$ in the right-hand side of (1.11). In the sequel, we prove the inequality (1.11) by flattening the boundary $\partial\Omega$ by rescaling and applying the Sobolev inequality around $\Omega_{x_0,r}$.

\begin{prop}
Let $\Omega$ be a uniformly $C^{1}$-domain in $\mathbf{R}^{n}$, $n\geq 2$.
Let $p>n$. Then, there exist constants $r_0$ and $C$ such that 
 \begin{equation*}
||\varphi||_{L^{\infty}(\Omega_{x_0,r})}\leq Cr^{-n/p}\left(||\varphi||_{L^{p}(\Omega_{x_0,2r})}+ r||\nabla \varphi||_{L^{p}(\Omega_{x_0,2r})}\right)\quad \textrm{for}\ \varphi\in W^{1,p}_{\textrm{loc}}(\bar{\Omega}),   \tag{A.2}
\end{equation*}
and $x_0\in \Omega$, $r\leq r_{0}$ satisfying $d_{\Omega}(x_0)<r$.
\end{prop}

From (A.1) and (A.2), for all $x_0\in \Omega$ and $r\leq r_0$, the inequality (1.11) follows.

\begin{lem}
Let $\Omega$ be a uniformly $C^{1}$-domain in $\mathbf{R}^{n}$, $n\geq 2$.
Let $p>n$. Then, the inequality (1.11) holds for all $\varphi\in W^{1,p}_{\textrm{loc}}(\bar{\Omega})$, $x_0\in \Omega$ and $r\leq r_0$ with the constant $C_{I}$ independent of $x_0$ and $r\leq r_0$ where $r_0$ is the constant in Proposition A.1.
\end{lem}

\begin{proof}
Take arbitrary points $x_0\in \Omega$ and $r\leq r_{0}$. If $d_{\Omega}(x_0)\geq r$, apply (A.1) to get (1.11) with the constant $C_{s}$. If $d_{\Omega}(x_0)< r$, we apply (A.2) for (1.11). 
\end{proof}

\begin{proof}[Proof of Proposition A.1]
We argue by contradiction.  Suppose on the contrary that the inequality (A.2) were false for any choice of constants $r_0$ and $C$. Then, there would exist sequences of points $\{x_m\}_{m=1}^{\infty}\subset \Omega$, $r_m\downarrow 0$ and a sequence of functions $\{\varphi_{m}\}_{m=1}^{\infty}\subset W^{1,p}_{\textrm{loc}}(\bar{\Omega})$ such that 
\begin{equation*}
||\varphi_{m}||_{L^{\infty}(\Omega_{x_m,r_m})}
>mr_{m}^{-n/p}\left( ||\varphi_{m}||_{L^{p}(\Omega_{x_m,2r_m})}+r_m||\nabla \varphi_{m}||_{L^{p}(\Omega_{x_m,2r_m})}\right).
\end{equation*}
Divide the both sides by $M_{m}=||\varphi_{m}||_{L^{\infty}(\Omega_{x_m,r_m})}$ and observe that $\tilde{\varphi}_{m}=\varphi_{m}/M_{m}$ satisfies 
\begin{equation*}
||\tilde{\varphi}_{m}||_{L^{\infty}(\Omega_{x_m,r_m})}=1,\quad r_{m}^{-n/p}\left(||\tilde{\varphi}_{m}||_{L^{p}(\Omega_{x_m,2r_m})}+r_m||\nabla \tilde{\varphi}_{m}||_{L^{p}(\Omega_{x_m,2r_m})}\right)<1/m.
\end{equation*}
Since the points $\{x_m\}_{m=1}^{\infty}\subset \Omega$ accumulate to the boundary by $d_m=d_{\Omega}(x_m)<r_m\downarrow 0$, by rotation and translation of $\Omega$, we may assume $x_{m}=(0,d_m)$. Set $c_m=d_m/r_m<1$. By choosing a subsequence of $\{c_m\}_{m=1}^{\infty}$, we may assume $c_m\to c_0$ as $m\to\infty$ for $c_0\leq 1$. In the sequel, we rescale the domain $\Omega$ around the point $ x_m\in \Omega$. Since $\Omega$ has a uniformly $C^{1}$-boundary, there exists  uniform constants $\alpha, \beta, K$ and $C^{1}$-function $h$ such that the neighborhood of the origin is represented by 
\begin{equation*}
\Omega_{\textrm{loc}}=\{x\in \mathbf{R}^{n}\ |\ h(x')<x_n<h(x')+\beta, |x'|<\alpha\},
\end{equation*}
where $h$ satisfies $h(0)=0$, $\nabla^{'}h(0)=0$ and $||h||_{C^{1}(B^{n-1}_{0}(\alpha))}\leq K$. Here, $B^{n-1}_{0}(\alpha)$ denotes the $n-1$-dimensional open ball centered at the origin with radius $\alpha$.

We rescale $\tilde{\varphi}_{m}$ around $x_m$ by 
\begin{equation*}
\phi_{m}(x)=\tilde{\varphi}_{m}(x_m+r_mx)\quad \textrm{for}\ x\in \Omega^{m},
\end{equation*}
where $\Omega^{m}=\{x\in \mathbf{R}^{n}\ |\ x_m+r_mx\in \Omega \}$. Then, the rescaled domain $\Omega^m$ expands to a half space $\mathbf{R}^{n}_{+,-c_0}=\{x\in \mathbf{R}^{n}\ |\ x_n>-c_0  \}$. In fact, $\Omega_{\textrm{loc}}$ is rescaled to 
\begin{equation*}
\Omega_{\rm{loc}}^{m}=\left\{x\in \mathbf{R}^{n}\ \Bigg|\ h_{m}(x')-c_m<x_n<h_{m}(x')-c_m+\frac{\beta}{r_m}, |x'|<\frac{\alpha}{r_m}   \right\},
\end{equation*}
where $h_{m}(x')=h(r_mx')/r_m$. The function $h_{m}$ and $\nabla h_m$ converges to zero locally uniformly in $\mathbf{R}^{n-1}$. Thus, $\Omega^{m}_{\textrm{loc}}$ expands to $\mathbf{R}^{n}_{+,-c_0}$. Since $\Omega^{m}_{0,1}=B_{0}(1)\cap \Omega^{m}$ may not be a $C^{1}$-domain on the intersection $\partial\Omega^{m}\cap B_{0}(1)$, we take a $C^{1}$-bounded domain $U^{m}$ so that $\Omega^{m}_{0,1}\subset U^{m}\subset \Omega^{m}_{0,2}$ and the $C^{1}$-regularity of $\partial U^{m}$ is uniformly bounded for $m\geq 1$. Since the $C^{1}$-norm of $h_{m}$ is locally uniformly bounded for $m\geq 1$ in $\mathbf{R}^{n-1}$, we are able to take such the $C^{1}$-domain $U^{m}$. 

Now, we apply the Sobolev inequality for $\phi_{m}$ in $U^{m}$ to get 
\begin{equation*}
||\phi_{m}||_{L^{\infty}(U^{m})}\leq C_{s}'||\phi_{m}||_{W^{1,p}(U^m)}
\end{equation*}
with the constant $C_{s}'$. The constant $C_{s}'$ depends on $m\geq 1$ but is bounded for all $m\geq 1$ since the $C^{1}$-regularity of $\partial U^{m}$ is uniformly bounded. Since the estimates for $\tilde{\varphi}_m$ are inherited to 
\begin{equation*}
||\phi_{m}||_{L^{\infty}(\Omega^{m}_{0,1})}=1,\quad ||\phi_{m}||_{W^{1,p}(\Omega^{m}_{0,2})}<1/m,
\end{equation*}
it follows that 
\begin{align*}
1=||\phi_{m}||_{L^{\infty}(\Omega^{m}_{0,1})}
&\leq ||\phi_{m}||_{L^{\infty}(U^{m})}  \\
&\leq C_{s}'||\phi_{m}||_{W^{1,p}(U^m)}\\
&\leq C_{s}'||\phi_{m}||_{W^{1,p}(\Omega^{m}_{0,2})}<C_{s}'/m\to 0 \quad \textrm{as}\ m\to\infty.
\end{align*}
We reached a contradiction. The proof is now complete. 
\end{proof}

\end{document}